\documentclass[twoside,11pt]{amsart} 
\usepackage{epsfig}
\usepackage{graphicx}
\usepackage{latexsym} 
\usepackage{amsmath} \usepackage{amssymb}
\usepackage{cancel}



\newtheorem{theorem}{Theorem}[section]
\newtheorem{lemma}[theorem]{Lemma}
\newtheorem{corollary}[theorem]{Corollary}
\newtheorem{remark}[theorem]{Remark}

\newdimen\epsfxsize
\newdimen\epsfysize

\newcommand {\smallgap}     {\makebox[0.025 in]{}}   
\newcommand {\gap}     {\makebox[0.075 in]{}}   
\newcommand {\biggap}     {\makebox[0.15 in]{}}   
\newcommand {\st}      {\gap : \gap}

\newcommand {\fto}     {\longrightarrow}

\newcommand {\set}[1]  {\left\{ {#1} \right\}}   
\newcommand {\ord}[1]  {{#1}^{\rm th}}

\newcommand{\pml}[1] {{#1}^\#}

\newcommand{\Z}     {{\mathbb Z}}
\newcommand{\Zmod}  {\Z \bmod \pml{p_k}}
\newcommand{\Zmodp}  {\Z \bmod \pml{p_{k+1}}}

\newcommand{\Rat}[2]  {w_{{#2},{#1}}}   

\newcommand{\pgap}   {{\mathcal G}}

\newcommand{\MJP}[2] {\left. M_{{#1}} \right|_{{#2}}}
\newcommand{\AJP}[2] {\left. \Lambda_{{#1}} \right|_{{#2}}}
\newcommand{\wgp}[2] {\left. \bar{w}_{{#1}} \right|_{{#2}}}

\newcommand{\Bi}[2]{\left( \begin{array}{c}{#1} \\ {#2} \end{array}\right)}
\newcommand{\lil}{\scriptstyle}

\begin{document}

\title{Combinatorics of the gaps between primes}


\date{2 Oct 2015}

\author{Fred B. Holt with Helgi Rudd}
\address{fbholt62@gmail.com ; 5520 - 31st Ave NE \#500, Seattle, WA 98105 -
presented at Connections in Discrete Mathematics -- June 2015, https://sites.google.com/site/connectionsindiscretemath/ }

\begin{abstract}
A few years ago we identified a recursion that works directly with the gaps among
the generators in each stage of Eratosthenes sieve.  This recursion provides explicit
enumerations of sequences of gaps among the generators, which sequences are known as 
constellations.

The populations of gaps and constellations across stages of Eratosthenes sieve
are modeled exactly by discrete dynamic systems.  These models and their asymptotic 
behaviors provide evidence on a number of
open problems regarding gaps between prime numbers.  

For Eratosthenes sieve we show that the analogue of Polignac's conjecture is true:
every gap $g=2k$ does occur in the sieve, and its asymptotic population supports 
the estimates made in Hardy and Littlewood's Conjecture B.

A stronger form of Polignac's conjecture also holds for the sieve:  for any gap $g=2k$,
every feasible constellation $g,g,\ldots,g$ occurs; these constellations correspond to
consecutive primes in arithmetic progression.

The models also provide evidence toward resolving a series of 
questions posed by Erd{\" o}s and Tur{\'a}n.  

\end{abstract}

\maketitle

\section{Introduction}
 
We work with the prime numbers in ascending order, denoting the
$\ord{k}$ prime by $p_k$.  Accompanying the sequence of primes
is the sequence of gaps between consecutive primes.
We denote the gap between $p_k$ and $p_{k+1}$ by
$g_k=p_{k+1}-p_k.$
These sequences begin
$$
\begin{array}{rrrrrrc}
p_1=2, & p_2=3, & p_3=5, & p_4=7, & p_5=11, & p_6=13, & \ldots\\
g_1=1, & g_2=2, & g_3=2, & g_4=4, & g_5=2, & g_6=4, & \ldots
\end{array}
$$

A number $d$ is the {\em difference} between prime numbers if there are
two prime numbers, $p$ and $q$, such that $q-p=d$.  There are already
many interesting results and open questions about differences between
prime numbers; a seminal and inspirational work about differences
between primes is Hardy and Littlewood's 1923 paper \cite{HL}.

A number $g$ is a {\em gap} between prime numbers if it is the difference
between consecutive primes; that is, $p=p_i$ and $q=p_{i+1}$ and
$q-p=g$.
Differences of length $2$ or $4$ are also gaps; so open questions
like the Twin Prime Conjecture, that there are an infinite number
of gaps $g_k=2$, can be formulated as questions about differences
as well.

A {\em constellation among primes} \cite{Riesel} is a sequence of consecutive gaps
between prime numbers.  Let $s=g_1 g_2 \cdots g_k$ be a sequence of $k$
numbers.  Then $s$ is a constellation among primes if there exists a sequence of
$k+1$ consecutive prime numbers $p_i p_{i+1} \cdots p_{i+k}$ such
that for each $j=1,\ldots,k$, we have the gap $p_{i+j}-p_{i+j-1}=g_j$. 

We do not study the gaps between primes directly.  Instead, we study the cycle of gaps
$\pgap(\pml{p})$ at each stage of Eratosthenes sieve.  Here, $\pml{p}$ is the
{\em primorial} of $p$, which is the product of all primes from $2$ up to and including $p$.
$\pgap(\pml{p})$ is the cycle of gaps among the generators of $\Z \bmod \pml{p}$.
These generators and their images through the counting numbers are the candidate primes
after Eratosthenes sieve has run through the stages from $2$ to $p$.  All of the remaining primes
are among these candidates.

There is a substantial amount of structure preserved in the cycle of gaps from
one stage of Eratosthenes sieve to the next, from $\pgap(\pml{p_k})$ to 
$\pgap(\pml{p_{k+1}})$.  This structure is sufficient to enable us to give exact counts for gaps and for 
sufficiently short constellations in $\pgap(\pml{p})$ across all stages of the sieve.


\subsection{Some conjectures and open problems regarding gaps between primes.}
Open problems regarding gaps and constellations between prime numbers include
the following.
\begin{itemize}
\item {\em Twin Prime Conjecture} - There are infinitely many pairs of consecutive primes with
gap $g=2$.
\item {\em Polignac's Conjecture} - For every even number $2n$, there are infinitely many pairs of 
consecutive primes with gap $g=2n$.
\item {\em HL Conjecture B} - From page 42 of Hardy and Littlewood \cite{HL}:  for any even $k$, the number of prime pairs $q$ and $q+k$ such that $q+k < n$ is approximately
$$ 2 C_2 \frac{n}{(\log n)^2} \prod_{p \neq 2, \; p | k} \frac{p-1}{p-2}.$$
\item {\em CPAP conjecture} - For every $k>2$, there exist infinitely many sets of $k$ consecutive
primes in arithmetic progression.
\end{itemize}

These problems and others regarding the gaps and differences among primes are
usually approached \cite{GranICM, GPY} through sophisticated probabilistic models, rooted in the prime number theorem.

Here we study the combinatorics of the cycle of gaps $\pgap(q \cdot \pml{p})$ as they relate to $\pgap(\pml{p})$.  
To obtain our results below, we look at the Chinese Remainder Theorem in the context of cycles of gaps,
supplemented by the more explicit arithmetics from the recursion between cycles.
Seminal works for our studies include \cite{HL, HW, Cramer}.  
Several estimates on gaps derived from these models
have been corroborated computationally.  These computations have
addressed the occurrence of twin primes \cite{Brent3, NicelyTwins, PSZ, IJ, JLB},
and some have corroborated the estimates in Conjecture B for other gaps
\cite{Brent, GranRaces}.

Work on specific constellations among primes includes the study of prime quadruplets
\cite{HL, quads},
which corresponds to the constellation $2,4,2$.  This is two pairs of twin primes separated
by a gap of $4$, the densest possible occurrence of primes in the large. 
The estimates for prime quadruplets have also been supported computationally
\cite{NicelyQuads}. 

We do not resolve any of the open problems as stated above for gaps between primes.
However, we are able to resolve their analogues for gaps in the stages of Eratosthenes sieve.
Through our work below we prove that analogues for all of the above conjectures hold true
for Eratosthenes sieve.

These results are deterministic, not probabilistic.  We develop a population model below
that describes the growth of the populations of various gaps in the cycle of gaps, across
the stages of Eratosthenes sieve.

\renewcommand\arraystretch{2}
\begin{table}
\begin{center}
\begin{tabular}{|l|c|p{2.6in}|} \hline
 & Primes & \multicolumn{1}{c|}{Eratosthenes sieve} \\ \hline
 Twin primes conjecture & {\em open} & $n_{2,1}(\pml{p}) = \prod_{2<q \le p} (q-2)$ \\ \hline
 Polignac's conjecture &  {\em open} & Every gap $g=2k$ occurs infinitely often. \\
 HL Conjecture B & {\em open} & $\frac{n_{2k,1}(\pml{p})}{n_{2,1}(\pml{p})} \fto \prod_{q>2, \; q | k} \frac{p-1}{p-2}$ \\ \hline
 CPAP conjecture & {\em open} & Every feasible constellation $gg\ldots g$ occurs, and its asymptotic
 relative population depends only on the prime factors of $g$. \\  \hline
\end{tabular}
\caption{\label{ConjTable} For Eratosthenes sieve, we can establish the analogues for
several open conjectures about gaps between primes. $n_{g,1}(\pml{p})$ denotes the population of
the gap $g$ among the generators of $\Z \bmod \pml{p}$.}
\end{center}
\end{table}
\renewcommand\arraystretch{1}

All gaps between prime numbers arise in a cycle of gaps.  To connect our results 
to the desired results on gaps between primes, we would need to better understand how
gaps survive later stages of the sieve, to be affirmed as gaps between
primes.

\section{The cycle of gaps}
After the first two stages of Eratosthenes sieve, we have removed the multiples of $2$ and $3$.
The candidate primes at this stage of the sieve are 
$$(1),5,7,11,13,17,19,23,25,29,31,35,37,41,43,\ldots$$
We investigate the structure of these sequences of candidate primes
by studying the cycle of gaps in the fundamental cycle.

For example, for the candidate primes listed above, the first gap from $1$ to $5$ is $g=4$,
the second gap from $5$ to $7$ is $g=2$, then $g=4$ from $7$ to $11$, and so on.
The {\em cycle of gaps} $\pgap(\pml{3})$ is $42$.  To reduce visual clutter, we write
the cycles of gaps as a concatenation of single digit gaps, reserving the use of commas to 
delineate gaps of two or more digits.

$$\pgap(\pml{3}) = 42, \gap {\rm with} \gap g_{1}=4 \gap {\rm and} \gap 
g_{2}=2.$$

Advancing Eratosthenes sieve one more stage, we identify $5$ as the next prime 
and remove the multiples of $5$ from the list of candidates, leaving us with
$$(1),7,11,13,17,19,23,29,31,37,41,43,47,49,53,59,61, \ldots$$
As illustrated in Figure~\ref{G5Fig}, we calculate the cycle of gaps at this stage to be
$\pgap(\pml{5}) = 64242462.$

\begin{figure}[t]
\centering
\includegraphics[width=5in]{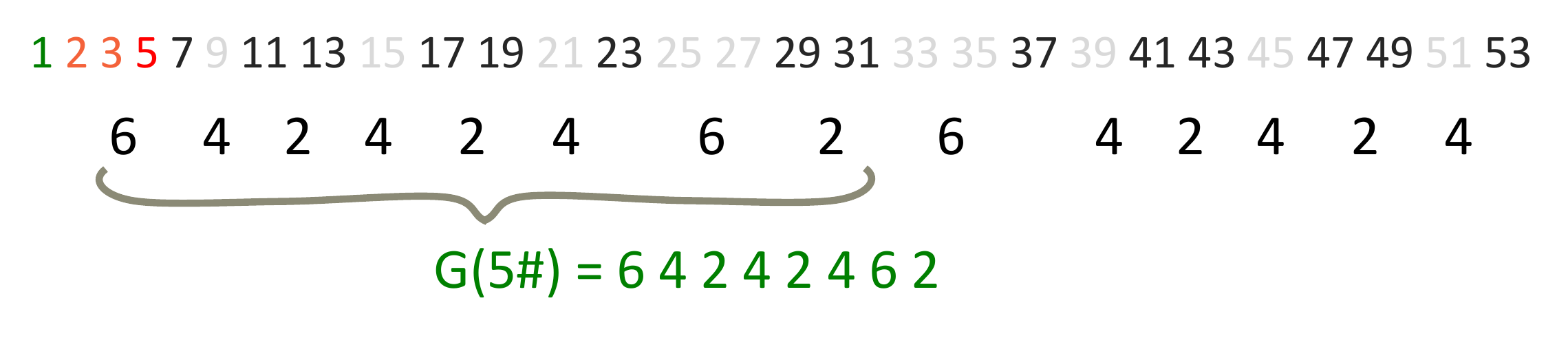}
\caption{\label{G5Fig} The cycle of gaps $\pgap(\pml{5})$ for Eratosthenes sieve
after multiples of $2$, $3$, and $5$ have been removed.  The first gap of $6$ goes from $1$ to the
next prime $7$. }
\end{figure}

\subsection{Recursion on the cycle of gaps}
There is a nice recursion which produces $\pgap(p_{k+1})$ directly from
$\pgap(p_k)$.  We concatenate $p_{k+1}$ copies of $\pgap(p_k)$, and
add together certain gaps as indicated by the entry-wise product
$p_{k+1}*\pgap(p_k)$.

\begin{lemma} \label{RecursLemma}
The cycle of gaps $\pgap(\pml{p_{k+1}})$ is derived recursively from 
$\pgap(\pml{p_k})$.
Each stage in the recursion consists of the following three steps:
\begin{itemize}
\item[R1.] Determine the next prime, $p_{k+1} = g_{1} + 1$.
\item[R2.] Concatenate $p_{k+1}$ copies of $\pgap(\pml{p_k})$.
\item[R3.] Add adjacent gaps as indicated by the elementwise product 
$p_{k+1}*\pgap(\pml{p_k})$:  let $i_1=1$ and add together $g_{i_1}+g_{i_1+1}$; then for 
$n=1,\ldots,\phi(N)$, add $g_{j}+g_{j+1}$ and let 
$i_{n+1}=j$ if the running sum of the concatenated gaps from $g_{i_n}$ to $g_j$ is
$p_{k+1}*g_{n}.$
\end{itemize}
\end{lemma}

\begin{proof}  
Let $\pgap(\pml{p_k})$ be the cycle of gaps for the stage of Eratosthenes sieve
after the multiples of the primes up through $p_k$ have been removed.  
Note that $\pgap(\pml{p})$ consists of $\phi(\pml{p})$ gaps that sum to $\pml{p}$.
We show that the recursion
R1-R2-R3 on $\pgap(\pml{p_k})$ produces the cycle of gaps for the next stage, corresponding
to the removal of multiples of $p_{k+1}$.

There is a natural one-to-one correspondence between the gaps in the cycle
of gaps $\pgap(\pml{p_k})$ and the generators of $\Z \bmod \pml{p_k}$.
For $j=1,\ldots, \phi(\pml{p_k})$ let
\begin{equation}\label{EqGen}
\gamma_{j} = 1 + \sum_{i=1}^j g_{i}.
\end{equation}
These $\gamma_{j}$ are the generators in $\Z \bmod \pml{p_k}$, with 
$\gamma_{\phi(\pml{p_k})} \equiv 1 \bmod \pml{p_k}$.
 
 The $\ord{j}$ candidate prime at this stage of the sieve is given by $\gamma_{j}$.

The next prime $p_{k+1}$ will be $\gamma_{1}$, since this will be the smallest
integer both greater than $1$ and coprime to $\pml{p_k}$.

The second step of the recursion extends our list of possible primes
up to $\pml{p_{k+1}}+1$, the reach of the fundamental cycle for $\pml{p_{k+1}}$.
For the gaps $g_{j}$ we extend the indexing on $j$ to cover these 
concatenated copies.  These $p_{k+1}$ concatenated copies of $\pgap(\pml{p_k})$
correspond to all the numbers from $1$ to $\pml{p_{k+1}}+1$ which are coprime
to $\pml{p_k}$.  For the set of generators of $\pml{p_{k+1}}$, we need only remove
the multiples of $p_{k+1}$. 

The third step removes the multiples of $p_{k+1}$.
Removing a possible prime amounts to
adding together the gaps on either side of this entry.
The only multiples of $p_{k+1}$ which remain in the copies of $\pgap(\pml{p_k})$
are those multiples all of whose prime factors are greater than $p_k$.
After $p_{k+1}$ itself, the next multiple to be removed will be $p_{k+1}^2$.

The multiples we seek to remove are given by $p_{k+1}$ times the generators
of $\Z \bmod \pml{p_k}$.  The consecutive differences between these will be given
by $p_{k+1} * g_{j}$, and the sequence $p_{k+1}*\pgap(\pml{p_k})$ suffices to cover
the concatenated copies of $\pgap(\pml{p_k})$.  We need not consider any fewer nor any
more multiples of $p_{k+1}$ to obtain the generators for $\pgap(\pml{p_{k+1}})$.

In the statement of R3, the index $n$ moves through the copy of
$\pgap(\pml{p_k})$ being multiplied by $p_{k+1}$, and the indices $\tilde{i}_n$
mark the index $j$ at which the addition of gaps occurs.
The multiples of $p_{k+1}$ in the sieve up through $\pml{p_{k+1}}$ are given by
$p_{k+1}$ itself and $p_{k+1}*\gamma_{j}$ for $j=1,\ldots,\phi(\pml{p_k})$.
The difference between successive multiples is $p_{k+1}*g_{j}$.
\end{proof}

%
\begin{figure}[t]
\begin{center}
\includegraphics[width=5in]{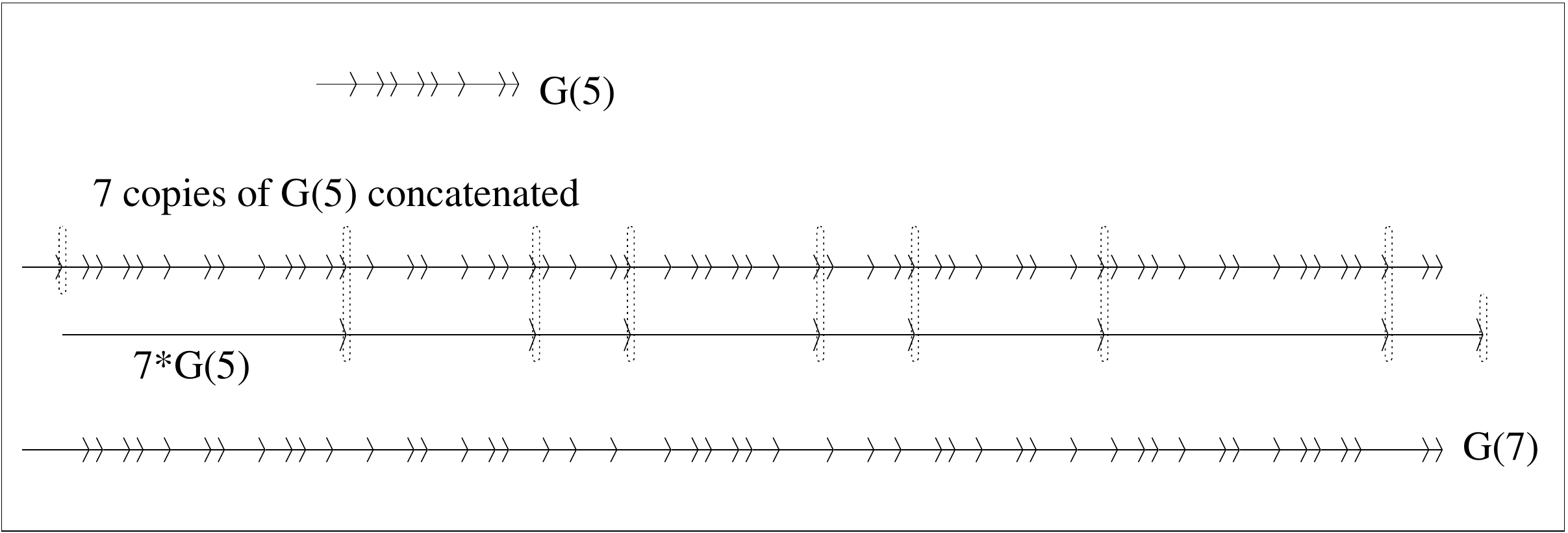}
\caption{\label{RecursFig} Illustrating the recursion that
produces the gaps for the next stage of Eratosthenes sieve.
The cycle of gaps $\pgap(\pml{7})$ is produced from $\pgap(\pml{5})$ by
concatenating $7$ copies, then adding the gaps indicated by
the element-wise product $7*\pgap(\pml{5})$.}
\label{default}
\end{center}
\end{figure}

We call the additions in step R3 the {\em closure} of the two adjacent gaps.  

The first closure in step R3 corresponds to noting the next prime number $p_{k+1}$.
The remaining closures in step R3 correspond to removing from the candidate 
primes the composite numbers whose smallest prime factor is $p_{k+1}$.
From step R2, the candidate primes have the form
 $\gamma + j\cdot \pml{p_k}$, for a generator $\gamma$ of $\Z \bmod \pml{p_k}$.

\noindent{\bf Example: $\pgap(\pml{7})$.}
As an example of the recursion,
we construct $\pgap(\pml{7})$ from $\pgap(\pml{5})=64242462$.
Figure~\ref{RecursFig} provides an illustration of this construction.

\begin{itemize}
\item[R1.] Identify the next prime, $p_{k+1}= g_1+1 = 7.$
\item[R2.] Concatenate seven copies of $\pgap(\pml{5})$:
$$\scriptstyle 64242462 \; 64242462 \; 64242462 \; 64242462\; 64242462 \; 64242462 \;64242462$$
\item[R3.] Add together the gaps after the leading $6$ and 
thereafter after differences of 
$ 7*\pgap(\pml{5}) = 42, 28, 14, 28, 14, 28, 42, 14 $:
$$\begin{array}{l} 
\pgap(\pml{7}) \; = \\
{\scriptstyle
  6+\overbrace{\scriptstyle 424246264242}^{42}+
 \overbrace{\scriptstyle 4626424}^{28}+\overbrace{\scriptstyle 2462}^{14}+
 \overbrace{\scriptstyle 6424246}^{28}+\overbrace{\scriptstyle 2642}^{14}+
 \overbrace{\scriptstyle 4246264}^{28}+\overbrace{\scriptstyle 242462642424}^{42}+
 \overbrace{\scriptstyle 62 \; }^{14}} \\
=  \; 
 {\it 10}, 2424626424 {\it 6} 62642 {\it 6} 46 {\it 8} 42424 {\it 8} 64 {\it 6} 24626 {\it 6} 
  4246264242, {\it 10}, 2
\end{array}
$$
The final difference of $14$ wraps around the end of the cycle,
 from the addition preceding the final $6$ to the 
addition after the first $6$.
\end{itemize}

\subsection{Every possible closure of adjacent gaps occurs exactly once}

When we apply the Chinese Remainder Theorem to this recursion on the cycle of gaps,
we derive very powerful combinatorial results.  The following Lemma~\ref{DelLemma}
is the foundation for developing the
discrete model for the populations of gaps and constellations across stages of Eratosthenes
sieve.  This is the reflection of the Chinese Remainder Theorem into this approach through the cycles
of gaps.

\begin{lemma}\label{DelLemma}
Each possible closure of adjacent gaps in the cycle $\pgap(\pml{p_k})$
occurs exactly once in the recursive construction of $\pgap(\pml{p_{k+1}}).$
\end{lemma}

A version of the {\em Chinese Remainder Theorem} sufficient for our needs is that 
the system
\begin{eqnarray*}
  \gamma = \gamma_1 \bmod N_1 \\ 
  \gamma = \gamma_2 \bmod N_2 
\end{eqnarray*}
with $\gcd(N_1,N_2)=1$ has a unique solution modulo $N_1 N_2$.

\begin{proof}
Each entry in $\pgap(\pml{p_k})$ corresponds to one of the generators
of $\Zmod$. The first gap $g_1$ corresponds to $p_{k+1}$, and
thereafter $g_j$ corresponds to $\gamma_{j} = 1+\sum_{i=1}^j g_i$.
These correspond in turn to unique combinations of nonzero
residues modulo the primes $2,3,\ldots,p_k$.  

In step R2, we concatenate $p_{k+1}$
copies of $\pgap(\pml{p_k})$.  For each gap $g_j$ in $\pgap(\pml{p_k})$
there are $p_{k+1}$ copies of this gap after step R2, corresponding to 
$$\gamma_{j} + i \cdot \pml{p_k} \biggap {\rm for} \gap i = 0, \ldots, p_{k+1}-1.$$
For each copy, 
the combination of residues for $\gamma_{j}$ modulo $2,3,\ldots,p_k$
is augmented by a unique residue modulo
$p_{k+1}$.  Exactly one of these has residue $0 \bmod p_{k+1}$, so we
perform $g_{j}+g_{j+1}$ for this copy and only this copy of $g_j$.
\end{proof}


\begin{remark}\label{EasyRmk}
The following results are easily established for $\pgap(\pml{p_k})$:

\begin{enumerate}
\item The cycle of gaps $\pgap(\pml{p_k})$ consists of $\phi(\pml{p_k})$ gaps
that sum to $\pml{p_k}$.
\item The first difference between closures is $p_{k+1}*(p_{k+1}-1)$,
which removes $p_{k+1}^2$ from the list of candidate primes.
\item The last entry in $\pgap(\pml{p_k})$ is always $2$. 
This difference goes from $-1$ to $+1$ in $\Zmod$.
\item The last difference $p_{k+1}*2$ between closures in step R3, wraps from
$-p_{k+1}$ to $p_{k+1}$ in $\Zmodp$.
\item Except for the final $2$, the cycle of differences is symmetric:
$$g_{j}=g_{\phi(\pml{p_k})-j}.$$
\item The constellation $gg\ldots g$ of length $j$ corresponds to $j+1$ consecutive primes in 
arithmetic progression (CPAP), and this constellation is feasible iff
$g = 0 \bmod p$ for all primes $p \leq j+1$.  
\item In $\pgap(\pml{p_{k+1}})$ there are at least two gaps of size $g=2p_k$.
\item The middle of the cycle $\pgap(\pml{p_k})$ is the sequence 
$$2^j,2^{j-1},\ldots,42424,\ldots,2^{j-1},2^j$$
in which $j$ is the smallest number such that $2^{j+1}>p_{k+1}$.
\end{enumerate}
\end{remark}

There is an interesting fractal character to the recursion.  To produce the next cycle of gaps
$\pgap(\pml{p_{k+1}})$ from the current one, $\pgap(\pml{p_k})$, we concatenate
$p_{k+1}$ copies of the current cycle, take an expanded copy of the current cycle, and close gaps
as indicated by that expanded copy.  In the discrete dynamic system that we develop below, we 
don't believe that all of the power in this self-similarity has yet been captured.

\section{Enumerating gaps and driving terms}\label{GapExSection}

By analyzing the application of Lemma~\ref{DelLemma} to the recursion, we can derive
exact counts of the occurrences of specific gaps and specific constellations across all
stages of the sieve.  

We start by exploring a few motivating examples, after which we 
describe the general process as a discrete dynamic system -- a population model with initial
conditions and driving terms.  Fortunately, although
the transfer matrix $M(p_k)$ for this dynamic system depends on the prime $p_k$, its eigenstructure
is beautifully simple, enabling us to provide correspondingly simple descriptions of the
asymptotic behavior of the populations.  In this setting, the populations are the numbers
of occurrences of specific gaps or constellations across stages of Eratosthenes sieve.

\subsection{Motivating examples}
We start with the cycle of gaps 
$$\pgap(\pml{5})=64242462$$
and study the persistence of its gaps and constellations through later stages of the sieve.

For a constellation $s$ of length $j$, we denote the number of occurrences of $s$ in the cycle of
gaps $\pgap(\pml{p})$ with the notation $n_{s,j}(\pml{p})$. 

A {\em driving term} of length $j+i$ for a constellation $s$ is 
a constellation of length $j+i$ which upon $i$ specific closures produces the constellation $s$.  For example,
the constellations $4 \; 2$ and $2 \; 4$ are driving terms of length $2$ for the gap $6$; and
the constellation $4 \gap 2 \gap 4$ is a driving term of length $3$ for the constellation $6 \gap 4$ and
for the gap $8$.
We denote the number of driving terms for $s$ of length $j+i$ as $n_{s,j+i}(\pml{p})$. 

{\bf Enumerating the gaps $g=2$ and $g=4$.}  For the gaps $g=2$ and $g=4$, we start with
$n_{2,1}(\pml{5})= n_{4,1}(\pml{5}) = 3$.  There are no constellations of length $2$ or more in 
$\pgap(\pml{5})$ that would produce gaps $2$ or $4$, so there are no driving terms other than the
gaps themselves.

In forming $\pgap(\pml{7})$, in step R2 we create $7$ copies
of each of the three $2$'s.  In step R3, in each family of seven copies we close a $2$ on the left and we close 
a $2$ on the right.

Could the two closures occur on the same copy of a $2$?  We observe that in step R3, the
distances between closures is governed by the entries in $7*\pgap(\pml{5})$, so the
minimum distance between closures in forming $\pgap(\pml{7})$ is $7*2 = 14$.
Thus the two closures cannot occur on the same copy of a $2$.  In fact, for any gaps and constellations
of sum less than $14$, the two external closures must occur in separate copies.

So for each $g=2$ in $\pgap(\pml{5})$, in step R2 we create seven copies, and in step R3 we
close two of these seven copies, one from the left and one from the right.  The identical argument applies
to $g=4$.  So the populations of the gaps $g=2$ and $g=4$ are completely described by:
\begin{equation}\label{Eq2System}
\begin{array}{lcl}
n_{2,1}(\pml{p_{k+1}}) = (p_{k+1}-2) \cdot n_{2,1}(\pml{p_k}) & {\rm with} & 
 n_{2,1}(\pml{5}) = 3 \\
n_{4,1}(\pml{p_{k+1}}) = (p_{k+1}-2) \cdot n_{4,1}(\pml{p_k}) & {\rm with} & 
 n_{4,1}(\pml{5}) = 3 
\end{array}
\end{equation}
From this we see immediately that at every stage of Eratosthenes sieve, 
$n_{2,1}(\pml{p}) = n_{4,1}(\pml{p})$, and that the number of gaps $g=2$ in the
cycle of gaps $\pgap(\pml{p})$ grows superexponentially by
factors of $p-2$ as we increase the prime $p$ through the stages of the sieve.

{\bf Enumerating the gaps $g=6$ and its driving terms.}
For the gap $g=6$, we count $n_{6,1}(\pml{5})=2$.  In forming $\pgap(\pml{7})$, 
we will create seven copies of each of these gaps and close two of the copies for each initial gap.
However, we will also gain gaps $g=6$ from the closures of the constellations $s=24$ and $s=42$.

These constellations $s=24$ and $s=42$ are {\em driving terms} for the gap $g=6$.  These driving
terms are of length $2$.  We observe that these constellations do not themselves have driving terms.
For $s=24$, we initially have $n_{24,2}(\pml{5})=2$, and under the recursion that forms
$\pgap(\pml{7})$, we create seven copies of each constellation $s=24$, and we close 
{\em three} of these copies.  The left and right closures remove these copies from the system for
$g=6$, and the middle closure produces a gap $g=6$. 

We can express the system for the population of gaps $g=6$ as:
\begin{eqnarray}\label{Eq6System}
\nonumber \left[ \begin{array}{c} n_{6,1} \\ n_{6,2} \end{array}\right]_{\pml{p_{k+1}}}
& = & \left[ \begin{array}{c} n_{6,1} \\  n_{24,2}+n_{42,2} \end{array}\right]_{\pml{p_{k+1}}} \\
 &= & \left[\begin{array}{cc} p_{k+1}-2 & 1 \\ 0 & p_{k+1}-3 \end{array}\right]
\left[ \begin{array}{c} n_{6,1} \\ n_{6,2} \end{array}\right]_{\pml{p_k}}
\end{eqnarray}
with $n_{6,1}(\pml{5})=2$ and $n_{6,2}(\pml{5})=n_{24,2}(\pml{5})+n_{42,2}(\pml{5})=4$.

By the symmetry of $\pgap(\pml{p})$, we know that 
$n_{24,2}(\pml{p})=n_{42,2}(\pml{p})$ for all $p$, but the above approach of noting the
addition of the populations of these driving terms will help us develop the general form 
for the dynamic system.

How does the population of gaps $g=6$ compare to that for $g=2$?  
\begin{center}
\begin{tabular}{|c|ccc|} \hline
 & \multicolumn{3}{c|}{$n_{2,j}$ and $n_{6,j}$ for small primes} \\
 & $\pgap(\pml{5})$ & $\pgap(\pml{7})$ & $\pgap(\pml{11})$  \\ \hline
 $\begin{array}{r}
 g=2 \\ g=6 \\ \hline s=24 \\ s=42 \end{array}$ &
 $\begin{array}{r}
 3 \\ 2 \\  2 \\ 2 \end{array}$ &
 $\begin{array}{rl}
15 & = 5\cdot 3 \\
14 & = 5\cdot 2 + 4\\
8 & = 4\cdot 2 \\
8 & = 4\cdot 2 \end{array}$ &
 $\begin{array}{rl}
135 & = 9\cdot 15 \\
142 & = 9\cdot 14 + 16\\ 
64 & = 8\cdot 8 \\
64 & = 8\cdot 8 \end{array}$ \\ \hline
  \end{tabular}
\end{center}

In $\pgap(\pml{7})$, 
there are still more gaps $2$ than $6$'s, but the gap $g=6$ now has $16$ driving terms.
These driving terms help make $6$'s more numerous than
$2$'s in $\pgap(\pml{11})$.
Thereafter, both populations $n_{2,1}(\pml{p})$ and $n_{6,1}(\pml{p})$
 are growing by the factor $(p-2)$, and the gap $g=6$ has 
driving terms whose populations grow by the factor $(p-3)$.

{\bf Enumerating the gaps $g=8$ and its driving terms.}
For the gap $g=8$, we have no gaps $g=8$ in $\pgap(\pml{5})$; however, there are driving terms of length two $s=26$ and $s=62$, and in this case there is a driving term of length three $s=242$.  When the left-inner closure of $s=242$ occurs,
the driving term $62$ is produced, and the right-inner closure produces a copy of the driving term $26$.
No other constellations in $\pgap(\pml{5})$ sum to $8$.  So how will the population of
the gap $g=8$ evolve over stages of the sieve?

As we have seen with the gaps $g=2, \smallgap 4, \smallgap 6$, in forming $\pgap(\pml{p_{k+1}})$ 
each gap $g=8$ will initially generate $p_{k+1}$ copies in step R2 of which $p_{k+1}-2$ will survive step R3.  Each instance
of $s=26$ or $s=62$ will generate one additional gap $g=8$ upon the interior closure, two copies will
be lost from the exterior closures, and $p_{k+1}-3$ copies will survive step R3 of the recursion.

The driving term of length three, $s=242$, will add to the populations of the driving terms
of length $2$.
In forming $\pgap(\pml{7})$, we will
create seven copies of $s=242$ in step R2.  In step R3, for the seven copies of $s=242$, 
the two exterior closures increase the sum,
removing the resulting constellation as a driving term for $g=8$; the two interior closures create
driving terms of length two ($s=62$ and $s=26$), and three copies of $s=242$ survive intact.

\begin{center}
\begin{tabular}{|c|ccc|} \hline
 & \multicolumn{3}{c|}{$n_{8,j}$ for small primes} \\
 & $\pgap(\pml{5})$ & $\pgap(\pml{7})$ & $\pgap(\pml{11})$  \\ \hline
 $\begin{array}{r}
 g=8 \\  \hline s=26 \\ s=62 \\ s=242 \end{array}$ &
 $\begin{array}{r}
 0 \\ 1 \\  1 \\ 1 \end{array}$ &
 $\begin{array}{rl}
2 & = 5\cdot 0 + 2 \\
5 & = 4\cdot 1 + 1\\
5 & = 4\cdot 1 + 1 \\
3 & = 3\cdot 1 \end{array}$ &
 $\begin{array}{rl}
28 & = 9\cdot 2 + 10 \\
43 & = 8\cdot 5 + 3\\ 
43 & = 8\cdot 5 + 3 \\
21 & = 7\cdot 3 \end{array}$ \\ \hline
  \end{tabular}
\end{center}

We now state this action as a general lemma for any constellation, which 
includes gaps as constellations of length one.


\begin{lemma}\label{Lemma2p}
For $p_k \ge 3$, let $s$ be a constellation of sum $g$ and length $j$, such that $g < 2\cdot p_{k+1}$.

Then for each instance of $s$ in $\pgap(\pml{p_k})$, in forming $\pgap(\pml{p_{k+1}})$, 
in step R2 we create $p_{k+1}$ copies of this instance of $s$, and the $j+1$ closures in step R3 
occur in distinct copies.  

Thus, under the recursion at this stage of the sieve,
each instance of $s$ in $\pgap(\pml{p_k})$generates $p_{k+1}-j-1$ copies of $s$ in 
$\pgap(\pml{p_{k+1}})$; the interior closures generate $j-1$ constellations of sum $g$ and
length $j-1$; and the two exterior closures increase the sum of the resulting constellation in two 
distinct copies, removing these from being driving terms for the gap $g$.
\end{lemma}

The proof is a straightforward application of Lemma~\ref{DelLemma}, 
but we do want to emphasize the role that the condition $g < 2\cdot p_{k+1}$
plays.  In step R3 of the recursion, as we perform closures across the $p_{k+1}$ concatenated copies
of $\pgap(\pml{p_k})$, the distances between the closures is given by the elementwise product
$p_{k+1} * \pgap(\pml{p_k})$.  Since the minimum gap in $\pgap(\pml{p})$ is $2$, the minimum 
distance between closures is $2\cdot p_{k+1}$.  And the condition $g < 2\cdot p_{k+1}$ ensures 
that the closures will
therefore occur in distinct copies of any instance of the constellation in $\pgap(\pml{p_k})$
created in step R2.

The count in Lemma~\ref{Lemma2p} is scoped to instances of a constellation.
These instances may overlap, but the count still holds.  For example, the gap $g=10$ has a driving
term $s=424$ of length three.  In $\pgap(\pml{5})= 64242462$, the two occurrences of $s=424$
overlap on a $4$.  The exterior closure for one is an interior closure for the other.  The count given
in the lemma tracks these automatically.

We illustrate Lemma~\ref{Lemma2p} in Figure~\ref{SystemFig}.  

\begin{figure}[t]
\centering
\includegraphics[width=4.875in]{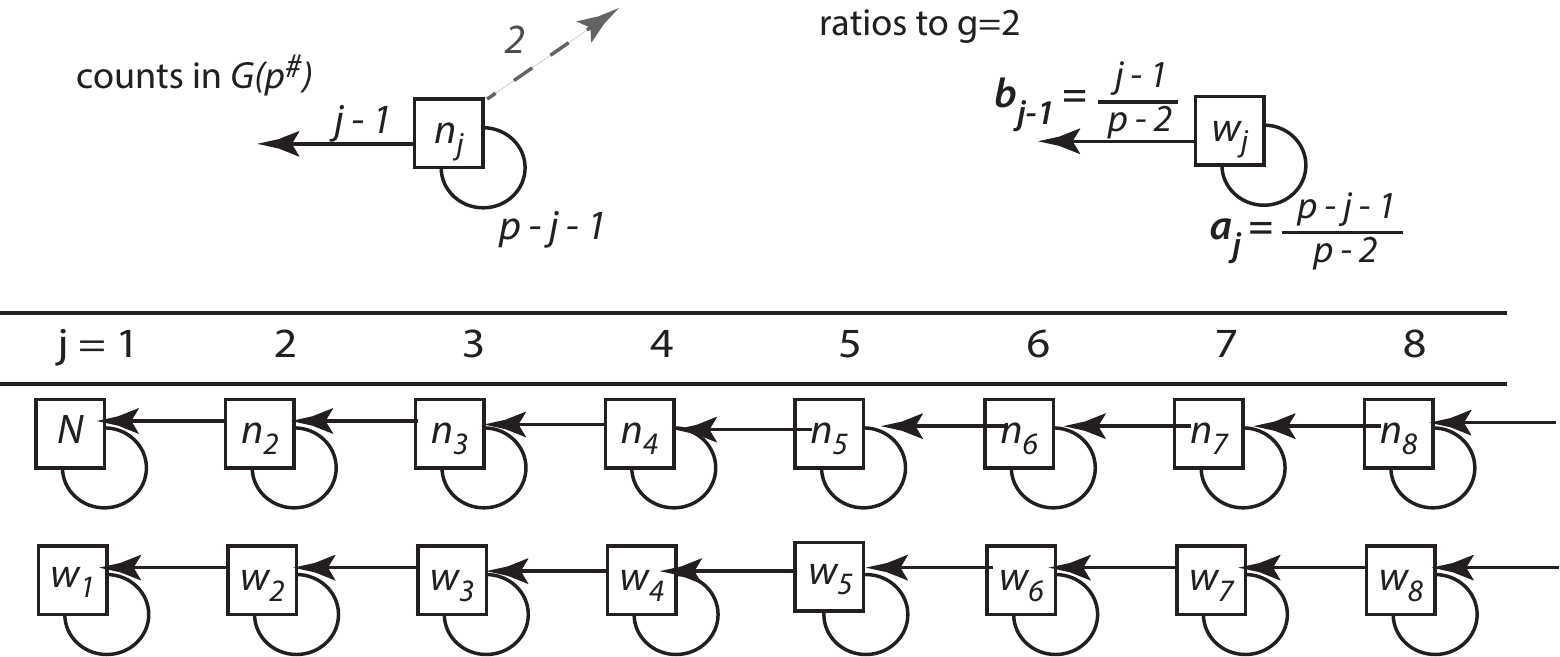}
\caption{\label{SystemFig} This figure illustrates the dynamic system of
Lemma~\ref{Lemma2p} through stages of the recursion for
$\pgap(\pml{p})$.
The coefficients of the system at each stage of the recursion
are independent of the specific gap and its driving terms.  
We illustrate the system for the recursive counts
$n_j$ for a gap and its driving terms.  Since the raw counts are superexponential, we take
the ratio $w_j$ of the count for each constellation to $n_{2,1}$ since $g=2$ has no driving terms.}
\end{figure}

A direct result of Lemma~\ref{DelLemma} and Lemma~\ref{Lemma2p} is an
exact description of the growth of the populations of various constellations across all stages
of Eratosthenes sieve.  (Keep in mind that a gap is a constellation of length $1$.)

\begin{theorem}\label{ThmGrowth}
If $s$ is any constellation in $\pgap(\pml{p_k})$ of length $j$ and sum $g < 2p_{k+1}$,
with $n_{s,j+1}(\pml{p_k})$ driving terms of length $j+1$ in $\pgap(\pml{p_k})$,
then 
$$n_{s,j}(\pml{p_{k+1}}) = (p_{k+1}-j-1)\cdot n_{s,j}(\pml{p_k}) + 1\cdot n_{s,j+1}(\pml{p_k}).$$
\end{theorem}

From this theorem, we note that the coefficients for the population model do not depend on the
constellation $s$.  The first-order growth of the population of every constellation $s$ of 
length $j$ and sum $g < 2p_{k+1}$ is given by the factor $p_{k+1}-j-1$.  This is independent
of the sequence of gaps within $s$.  

Although the asymptotic growth of all constellations of length $j$
is equal, the initial conditions and driving terms are important.
As we have seen above, constellations may differ
significantly in the populations of their driving terms.
Brent \cite{Brent} made analogous observations for single gaps ($j=1$).
His Table 2 indicates the importance of the lower-order effects in
estimating relative occurrences of certain gaps.

\subsection{Relative populations of $g=2, 6, 8, 10, 12$}
What can we say about the relative populations of the gaps $g=2, 6, 8, 10, 12$ over later stages of the
sieve?  The population of every gap grows by a factor of $p-2$.  The populations differ by the presence 
of driving terms of various lengths and by the initial conditions.

We proceed by normalizing each population by the population of the gap $g=2$.  
To compare the populations of any gap $g$ to the gap $2$ over later stages of the sieve, 
we take the ratio of the population of a gap to $n_{2,1}$.
\begin{equation}\label{EqDefW}
\Rat{1}{g}(\pml{p}) = \frac{ n_{g,1}(\pml{p})}{n_{2,1}(\pml{p})}.
\end{equation}
Letting $n_{g,j}(\pml{p})$ denote the number of all driving terms of sum $g$ and length $j$
in the cycle of gaps $\pgap(\pml{p})$, we can extend this definition to
\begin{equation*}
\Rat{j}{g}(\pml{p}) = \frac{ n_{g,j}(\pml{p})}{n_{2,1}(\pml{p})}.
\end{equation*}

\renewcommand\arraystretch{2}
These ratios for the gaps from $g=6$ to $g=12$ are given by the $4$-dimensional dynamic system:
\begin{eqnarray*}
\left[\begin{array}{c} \Rat{1}{g} \\ \Rat{2}{g} \\ \Rat{3}{g} \\ \Rat{4}{g} \end{array} \right] _{\pml{p_{k+1}}} & = &
\left[\begin{array}{cccc} 1 & \frac{1}{p_{k+1}-2} & 0 & 0 \\ 0 & \frac{p_{k+1}-3}{p_{k+1}-2} & \frac{2}{p_{k+1}-2} & 0 \\
 0 & 0 & \frac{p_{k+1}-4}{p_{k+1}-2} & \frac{3}{p_{k+1}-2} \\ 0 & 0 & 0 & \frac{p_{k+1}-5}{p_{k+1}-2} \end{array} \right] \cdot
\left[\begin{array}{c} \Rat{1}{g} \\ \Rat{2}{g} \\ \Rat{3}{g} \\ \Rat{4}{g} \end{array} \right]_{\pml{p_k}} \\
 & & \\
{\rm or} & & \\
 \wgp{g}{\pml{p_{k+1}}} & = &
   \MJP{1:4}{p_{k+1}} \cdot \wgp{g}{\pml{p_k}} 
\end{eqnarray*}
\renewcommand\arraystretch{1}

For this dynamic system, our attention turns to the $4 \times 4$ system matrix $M_{1:4}(p)$ and its
eigenstructure.  The notation $M_{j_1:J}$ provides the range of lengths of constellations over which we 
apply the system.  Anticipating our work below on constellations, we include the initial index $1$ in
$M_{1:J}$ to indicate that the system is modeling the populations of gaps.

Note that the system matrix depends on the prime $p$ (but not on the gap $g$), so that as we iterate, 
we have to keep track of this dependence.
$$
\wgp{g}{\pml{p_k}}  =  \MJP{1:4}{p_k} \cdot \MJP{1:4}{p_{k-1}} \cdots \MJP{1:4}{p_1} \wgp{g}{\pml{p_0}} 
$$

A simple calculation shows that we are in luck.  The eigenvalues of $M_{1:4}(p)$ depend on $p$
but the eigenvectors do not.  We write the eigenstructure of $M_{1:4}(p)$ as
\begin{equation*}
\begin{array}{rcccccc}
 \MJP{1:4}{p} & =  & R_4 & \cdot &  \left. \Lambda_{1:4} \right|_p & \cdot & L_4 \\
  & & & & & & \\
 & = &  \left[\begin{array}{rrrr} 1 & -1 & 1 & -1 \\ 0 & 1 & -2 & 3 \\ 0 & 0 & 1 & -3 \\ 0 & 0 & 0 & 1 \end{array}\right]
  & \cdot & \left[\begin{array}{rccc} 1 & 0 & 0 & 0 \\ 0 & \frac{p-3}{p-2} & 0 & 0  \\ 0 & 0 & \frac{p-4}{p-2} & 0 \\
  0 & 0 & 0 & \frac{p-5}{p-2} \end{array}\right] 
  & \cdot & \left[\begin{array}{rrrr} 1 & 1 & 1 & 1 \\ 0 & 1 & 2 & 3 \\ 0 & 0 & 1 & 3 \\ 0 & 0 & 0 & 1 \end{array}\right]
\end{array}
\end{equation*}
in which $\Lambda(p)$ is the diagonal matrix of eigenvalues, $R$ is the matrix of right eigenvectors,
and $L$ is the matrix of left eigenvectors, such that $L\cdot R = I$.  It is interesting that $L$ is an upper
triangular Pascal matrix and $R$ is an upper triangular Pascal matrix of alternating sign.

Since the eigenvectors do not depend on $p$, the iterative system simplifies:
\begin{eqnarray*}
\wgp{g}{\pml{p_k}} &  = & \MJP{1:4}{p_k} \cdot \MJP{1:4}{p_{k-1}} \cdots \MJP{1:4}{p_1} \wgp{g}{\pml{p_0}} \\
  & = & R_4 \cdot \AJP{1:4}{p_k} \cdot \AJP{1:4}{p_{k-1}} \cdots \AJP{1:4}{p_1} \cdot L_4 \wgp{g}{\pml{p_0}} 
\end{eqnarray*}
The dependence on $p$ leads to a product of diagonal matrices.  
 
Fixing $p_0$, we define $M^k_{1:4} = M_{1:4}(p_k) \cdots M_{1:4}(p_1)$:
\begin{equation*}
\begin{array}{rccc} M^k_{1:4}  = & R_4 & \AJP{1:4}{p_k} \cdots \AJP{1:4}{p_1} & L_4 \\
  & & & \\
  = & \left[\begin{array}{rrrr} \lil{1} & \lil{-1} & \lil{1} & \lil{-1} \\ \lil{0} & \lil{1} & \lil{-2} & \lil{3} \\ 
  \lil{0} & \lil{0} & \lil{1} & \lil{-3} \\ \lil{0} & \lil{0} & \lil{0} & \lil{1} \end{array}\right] 
 \cdot &  \left[\begin{array}{rccc} \lil{1} & \lil{0} & \lil{0} & \lil{0} \\ 
  \lil{0} & \lil{\prod \frac{p-3}{p-2}} & \lil{0} & \lil{0} \\ 
  \lil{0} & \lil{0} & \lil{\prod \frac{p-4}{p-2}} & \lil{0} \\    
  \lil{0} & \lil{0} & \lil{0} & \lil{\prod \frac{p-5}{p-2}} \end{array}\right]_{p_1}^{p_k} 
  & \cdot \left[\begin{array}{rrrr} \lil{1} & \lil{1} & \lil{1} & \lil{1} \\ \lil{0} & \lil{1} & \lil{2} & \lil{3} \\ 
  \lil{0} & \lil{0} & \lil{1} & \lil{3} \\ \lil{0} & \lil{0} & \lil{0} & \lil{1} \end{array}\right] 
  \end{array}
 \end{equation*}
 
 The eigenvalues are the products of the $a_j$ across ranges of primes:  
 $$\lambda_j^k = a_j^k = \prod_{p_1}^{p_k} \frac{p-j-1}{p-2}. $$
  
 For any gap $g < 2p_1$ that has driving terms of a maximum length of $4$, once we know the 
 initial populations in $\pgap(\pml{p_0})$, we can use the eigenstructure
 of $M^k_{1:4}$ to completely characterize the populations of $g$ and its driving terms in
 a very compact form.  Starting with the initial conditions
 $$
\wgp{g}{\pml{p_0}}  = \left[ \begin{array}{c} 
 w_{g,1} \\ w_{g,2} \\ w_{g,3} \\ w_{g,4} \end{array} \right]_{\pml{p_0}},
 $$
 we apply the left eigenvectors $L_4$ to obtain the coordinates relative to the basis
 of right eigenvectors $R_4$.  After this transformation, we can apply the actions of 
 the eigenvalues $\Lambda_{1:4}(p)$ directly to the individual right eigenvectors.  
\begin{eqnarray*}
\wgp{g}{\pml{p_k}} & = & M^k_{1:4} \cdot \wgp{g}{\pml{p_0}} \\
 & = & R_4 \cdot \AJP{1:4}{p_k} \cdots \AJP{1:4}{p_1} \cdot L_4 \cdot \wgp{g}{\pml{p_0}} 
  \gap = \gap R_4 \cdot \Lambda_{1:4}^k \cdot L_4 \cdot \wgp{g}{\pml{p_0}} \\
 & = & (L_{4,1} \cdot \wgp{g}{\pml{p_0}}) R_{4,1} \; + \; \lambda_2^k (L_{4,2} \cdot \wgp{g}{\pml{p_0}}) R_{4,2}
  \; + \cdots \\
  & = & \sum_j w_{g,j}(\pml{p_0}) e_1 \; + \; \lambda_2^k (L_{4,2} \cdot \wgp{g}{\pml{p_0}}) R_{4,2}
  \; + \cdots 
  \end{eqnarray*}
  
 Right away we observe that the asymptotic ratio $w_{g,1}(\infty)$ of the gap $g$ to the gap $2$ is the sum
 of the initial ratios of all of $g$'s driving terms.  We also observe that the ratio converges
 to the asymptotic value as quickly as $a_2^k \fto 0$.  While $a_3^k$ becomes small pretty
 quickly, the convergence of $a_2^k$ is slow. Figure~\ref{AjkFig} plots $a_2^k$ and $a_3^k$
 for $p_0=13$ up to $p \approx 3\cdot 10^{15}$.

\begin{figure}[t]
\centering
\includegraphics[width=5in]{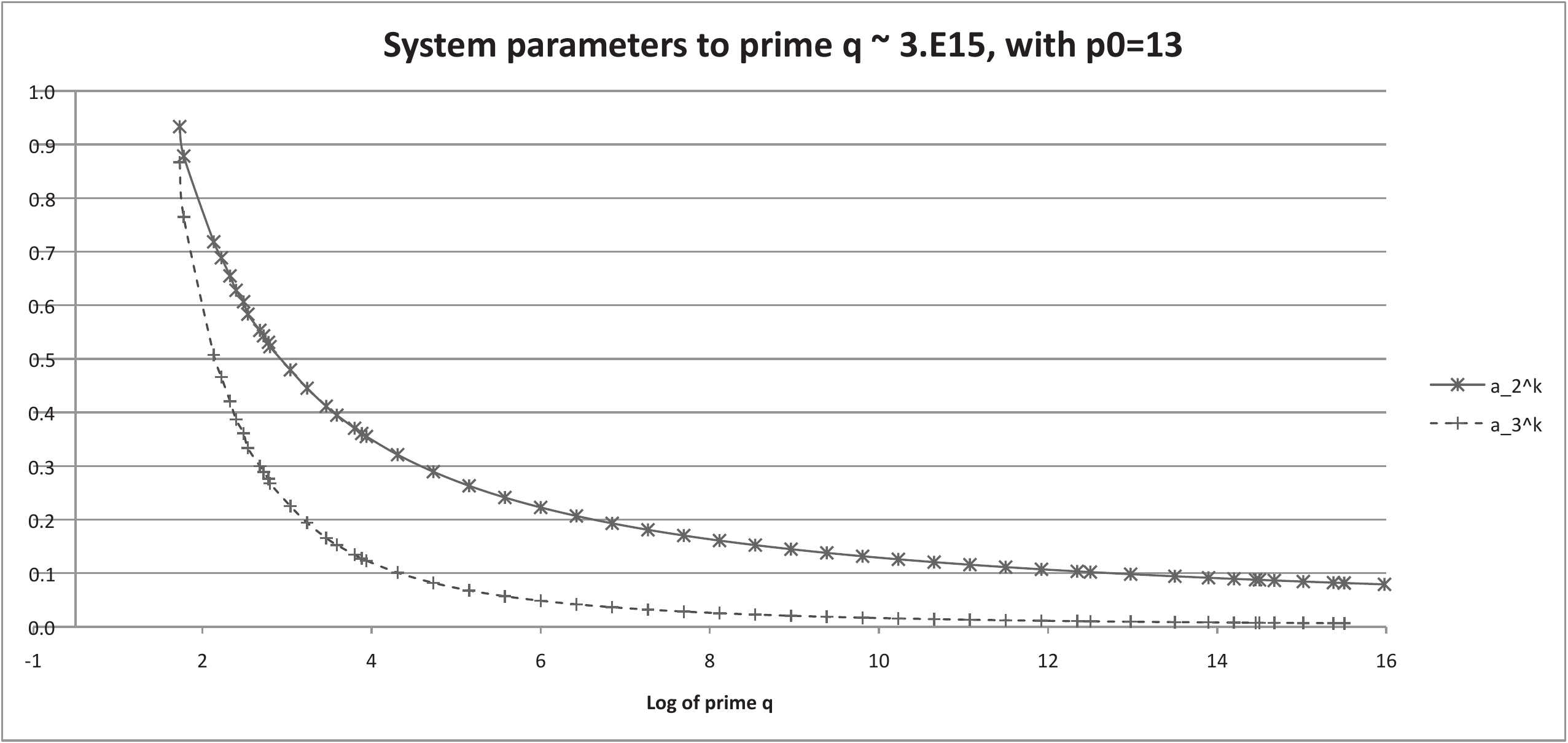}
\caption{\label{AjkFig} A graph of $a_2^k$ and $a_3^k$, with $p_0=13$, up to 
about $p \approx 3\cdot 10^{15}$. The dominant eigenvalue for $M_J$ is $1$, the second
eigenvalue is $a_2$ and the third $a_3$.  So the rate of convergence to the asymptotic 
ratio $w_{g,1}(\infty) = N_g / N_2$ is governed by how quickly $a_2^k \fto 0$. }
\end{figure}
 
 Let us apply this asymptotic analysis to all of the gaps that satisfy the required conditions
 in $\pgap(\pml{5})$.  Using
 $p_0=5$, we observe in $\pgap(\pml{5})=64242462$ that the gaps $g = 4, 6, 8, 10, 12$
 have driving terms of length at most $4$ and satisfy $g < 2 \cdot 7$.
 
 \begin{center}
 \begin{tabular}{|r|rrrr|r|} \hline
\multicolumn{6}{|c|}{$\pgap(\pml{5})=64242462$ } \\
gap & \multicolumn{4}{|c|}{$w_{g,j}(\pml{5})$} & $w_{g,1}(\infty)$ \\ 
 $g$ & $j=1$ & $2$ & $3$ & $4$ &  \\ \hline
 $4$ & $1$ & $0$ & $0$ & $0$ & $1$ \\
 $6$ & $2/3$ & $4/3$ & $0$ & $0$ & $2$ \\
 $8$ & $0$ & $2/3$ & $1/3$ & $0$ & $1$ \\
 $10$ & $0$ & $2/3$ & $2/3$ & $0$ & $4/3$ \\
 $12$ & $0$ & $0$ & $4/3$ & $2/3$ & $2$ \\ \hline
 \end{tabular}
\end{center}

To obtain the asymptotic ratio $w_{g,1}(\infty)$, we simply add together the initial 
ratios of all driving terms.  These results tell us that as quickly as $a_2^k \fto 0$,
the number of occurrences of the gap $g=6$ and the gap $g=12$ in the
cycle of gaps $\pgap(\pml{p})$ for Eratosthenes sieve each approach double the number
of gaps $g=2$.  Despite having driving terms of length two and of length three, the number of
gaps $g=8$ approaches the number of gaps $g=2$ in later stages of the sieve, and the
ratio of the number of gaps $g=10$ to the number of gaps $g=2$ approaches $4/3$.

These are {\em not} probabilistic estimates.  These ratios are based on the actual counts of the 
populations of these gaps and their driving terms across stages of Eratosthenes sieve.

\section{A model for populations across iterations of the sieve}\label{GapSection}
We now identify a general discrete dynamic system that provides exact counts
of a gap and its driving terms.  These raw counts grow
superexponentially, and so to better understand their behavior we take the ratio
of a raw count to the number of gaps $g=2$ at each stage of the sieve.
In the work above we created and examined this dynamic system for driving
terms up to length $3$.  Here we generalize this approach by considering driving
terms up to length $J$, for any $J$.

Fix a sufficiently large size $J$.
For any gap $g$ that has driving terms of lengths up to $j$, with $j \le J$, we form a vector
of initial values $\left. \bar{w}\right|_{p_{0}}$, whose $i^{\rm th}$ entry is the ratio
of the number of driving terms for $g$ of length $i$ in $\pgap(\pml{p_0})$
to the number of gaps $2$ in this cycle of gaps.

Generalizing our work for $J=4$ above, we model the population of the gap $g$ and its
driving terms across stages of Eratosthenes sieve as a discrete dynamic system.
\begin{eqnarray*}
\left. \bar{w}\right|_{\pml{p_k}} & = & M_{1:J}(p_k) \cdot \left. \bar{w}\right|_{\pml{p_{k-1}}} \\
 & = & \left[\begin{array}{cccccc}
1 & b_1 & 0 & \cdots & & 0 \\
0  & a_2 & b_2 & \ddots &  & 0 \\
 & 0 & a_3 & b_3 & \ddots & 0 \\
  \vdots & & \ddots & \ddots & \ddots & 0 \\
  0 & & \cdots & & a_{J-1} & b_{J-1} \\
  0 & & \cdots & & 0 & a_J \end{array} \right]_{p_k} \cdot \left. \bar{w}\right|_{\pml{p_{k-1}}} 
\end{eqnarray*}
in which
\begin{equation}\label{Eqaj}
a_j(p)  =  \frac{p-j-1}{p-2} \biggap {\rm and} \biggap b_j(p) = \frac{j}{p-2}.
\end{equation}

Iterating this discrete dynamic system from the initial conditions at $p_0$ up through $p_k$, we have
\begin{eqnarray*}
\left. \bar{w}\right|_{\pml{p_{k}}} & = & \left. M_{1:J} \right|_{p_{k}} \cdot \left. \bar{w}\right|_{\pml{p_{k-1}}} 
 \; = \left. M_{1:J} \right|_{p_{k}} \cdots \left. M_{1:J} \right|_{p_{1}} \cdot \left. \bar{w}\right|_{\pml{p_0}} \\
 & = & M_{1:J}^k  \cdot \left. \bar{w}\right|_{\pml{p_{0}}} 
\end{eqnarray*}
The matrix $M_{1:J}$ does not depend on the gap $g$.  It does depend on the prime $p_k$,
and we use the exponential notation $M_{1:J}^k$ to indicate the product of the $M$'s over
the indicated range of primes.

That $M_{1:J}$ does not depend on the gap $g$ is interesting.  This means that the recursion 
treats all gaps fairly.  The recursion itself is not biased toward certain gaps or constellations.
Once a gap has driving terms in $\pgap(\pml{p})$, the populations across all further stages of 
the sieve are completely determined.

With $M_{1:J}^k$ we can calculate the ratios $\Rat{j}{g}(p_k)$
for the complete system of gaps and their driving terms, relative to the
population of the gap $2$, for the cycle of gaps $\pgap(\pml{p_k})$ (here, $p_k$ is the $k^{\rm th}$
prime after $p_0$).  With $J=4$ we calculated above the ratios for $g=6, 8, 10, 12.$
For $g=30$, we need $J=8$.

Fortunately, we can completely describe the eigenstructure for $\left. M_{1:J}\right|_{p}$,
and even better -- {\em the eigenvectors for $M_{1:J}$ do not depend on the prime $p$}.  This means
that we can use the eigenstructure to provide a simple description of the behavior of this 
iterative system as $k \fto \infty$.

\subsection{Eigenstructure of $M_{1:J}$}
We list the eigenvalues, the left eigenvectors and the right eigenvectors for $M_{1:J}$, writing these
in the product form
$$ M_{1:J}  =  R \cdot \Lambda \cdot L $$
with $LR = I$.
For the general system $M_{1:J}$, the upper triangular entries of $R$ and $L$ are binomial coefficients,
with those in $R$ of alternating sign; and the eigenvalues are the $a_j$ defined in Equation~\ref{Eqaj}
above.
\begin{eqnarray*}
R_{ij} & = & \left\{ \begin{array}{cl}
(-1)^{i+j}\Bi{j-1}{i-1} & {\rm if} \; i \le j \\ & \\
0 & {\rm if} \; i > j \end{array} \right. \\
& & \\
\Lambda & = & {\rm diag}(1, a_2, \ldots, a_J) \\
 & & \\
L_{ij}  & = & \left\{ \begin{array}{ll}
\Bi{j-1}{i-1} & {\rm if} \; i \le j \\ & \\
0 & {\rm if} \; i > j \end{array} \right. 
\end{eqnarray*}

For any vector $\bar{w}$, multiplication by the left eigenvectors 
(the rows of $L$) yields the 
coefficients for expressing this vector of initial conditions over the basis given by the 
right eigenvectors (the columns of $R$):
$$ \bar{w} =  (L_{1 \cdot} \bar{w}) R_{\cdot 1} + \cdots + (L_{J \cdot} \bar{w}) R_{\cdot J}
$$

\begin{lemma}\label{AsymLemma}
Let $g$ be a gap and $p_0$ a prime such that $g < 2 p_1$.
In $\pgap(\pml{p_0})$ let the initial ratios for $g$ and its driving terms be given by
$\wgp{g}{\pml{p_0}}$.
Then the ratio of occurrences of this gap $g$
to occurrences of the gap $2$ in $\pgap(\pml{p})$ as $p \fto \infty$ converges to
the sum of the initial ratios across the gap $g$ and all its driving terms:
$$
\Rat{1}{g}(\infty) = L_{1 \cdot} \wgp{g}{\pml{p_0}} = \sum_j \left. \Rat{j}{g} \right|_{\pml{p_0}}.
$$
\end{lemma}

\begin{proof}
Let $g$ have driving terms up to length $J$.  Then the ratios $\wgp{g}{\pml{p}}$ 
are given by  the iterative linear system
$$
\wgp{g}{\pml{p_k}} = M_J^k \cdot  \wgp{g}{\pml{p_0}}.
$$
From the eigenstructure of $M_J$, we have
$$
\wgp{g}{\pml{p_0}} = (L_1 \wgp{g}{\pml{p_0}})R_1 + (L_2 \wgp{g}{\pml{p_0}})R_2 
+ \cdots + (L_J \wgp{g}{\pml{p_0}})R_J ,
$$
and so
\begin{eqnarray} \label{EqDynSys}
M_J^k \wgp{g}{\pml{p_0}} &=& 
  (L_1 \wgp{g}{\pml{p_0}})R_1 + a_2^k (L_2 \wgp{g}{\pml{p_0}})R_2  + \cdots \\ 
 & &  \cdots + a_J^k (L_J \wgp{g}{\pml{p_0}})R_J.  \nonumber
\end{eqnarray}
We note that $L_{1 \cdot} = [1 \cdots 1]$, $\lambda_1 = 1$, and $R_{\cdot 1} = e_1$;
that the other eigenvalues $a_j^k \fto 0$ with $a_j^k > a_{j+1}^k$.   Thus as $k \fto \infty$
the terms on the righthand side decay to $0$ except for the first term, establishing the result.
\end{proof}

With Lemma~\ref{AsymLemma} and the initial values in $\pgap(\pml{13})$, 
we can calculate the asymptotic ratios
of the occurrences of the gaps $g = 6, 8, \ldots, 32$ to the gap $g=2$. 
These are the gaps that satisfy the condition $g < 2 \cdot 17$.
We tabulate these initial values in Table~\ref{G13Table}, and the asymptotic ratios.

In Table~\ref{G13Table} we use $p_0=13$ for our initial conditions since the prime 
$p=13$ is the first prime 
for which the conditions of Theorem~\ref{ThmGrowth} 
are satisifed for the next primorial $g=\pml{5} = 30$.  

\renewcommand\arraystretch{0.8}
\begin{table}
\begin{center}
\begin{tabular}{|r|rrrrrrrrr|r|} \hline
gap & \multicolumn{9}{c|}{ $n_{g,j}(13)$: driving terms of length $j$ in $\pgap(\pml{13})$ } & $w_{g,1}(\infty)$ \\ [1 ex]
$g$ & $j=1$ & $2$ & $3$ & $4$ & $5$ & $6$ & $7$ & $8$ & $9$ & \\ \hline 
 $\lil 2, \; 4$ & $\lil 1485$ & & & & & & & & & $\lil{1}$ \\
 $\lil 6$ & $\lil 1690$ & $\lil 1280$ & & & & & & & & $\lil{2}$ \\
 $\lil 8$ & $\lil 394$ & $\lil 902$ & $\lil 189$ & & & & & & &  $\lil{1}$ \\
 $\lil 10$ & $\lil 438$ & $\lil 1164$ & $\lil 378$ & & & & & & & $\lil{4/3}$ \\
 $\lil 12$ & $\lil 188$ & $\lil 1276$ & $\lil 1314$ & $\lil 192$ & & & & & & $\lil{2}$ \\
 $\lil 14$ & $\lil 58$ & $\lil 536$ & $\lil 900$ & $\lil 288$ & & & & & & $\lil{6/5}$ \\
 $\lil 16$ & $\lil 12$ & $\lil 252$ & $\lil 750$ & $\lil 436$ & $\lil 35$ & & & & &  $\lil{1}$ \\
 $\lil 18$ & $\lil 8$ & $\lil 256$ & $\lil 1224$ & $\lil 1272$ & $\lil 210$ & & & & & $\lil{2}$ \\
 $\lil 20$ & $\lil 0$ & $\lil 24$ & $\lil 348$ & $\lil 960$ & $\lil 600$ & $\lil 48$ & & & & $\lil{4/3}$ \\
 $\lil 22$ & $\lil 2$ & $\lil 48$ & $\lil 312$ & $\lil 784$ & $\lil 504$ & & & & & $\lil{10/9}$ \\
 $\lil 24$ & $\lil 0$ & $\lil 20$ & $\lil 258$ & $\lil 928$ & $\lil 1260$ & $\lil 504$ & & & & $\lil{2}$ \\
 $\lil 26$ & $\lil 0$ & $\lil 2$ & $\lil 40$ & $\lil 322$ & $\lil 724$ & $\lil 448$ & $\lil 84$ & & & $\lil{12/11}$ \\
 $\lil 28$ & $\lil 0$ & $\lil 0$ & $\lil 36$ & $\lil 344$ & $\lil 794$ & $\lil 528$ & $\lil 80$ & & & $\lil{6/5}$ \\
 $\lil 30$ & $\lil 0$ & $\lil 0$ & $\lil 10$ & $\lil 194$ & $\lil 1066$ & $\lil 1784$ & $\lil 816$ & $\lil 90$ & & $\lil{8/3}$ \\
 $\lil 32$ & $\lil 0$ & $\lil 0$ & $\lil 0$ & $\lil 12$ & $\lil 200$ & $\lil 558$ & $\lil 523$ & $\lil 172$ & $\lil 20$ & $\lil{1}$ \\ \hline
 \end{tabular}
\caption{ \label{G13Table} For small gaps $g$, 
this table lists the number of gaps and driving terms of length $j$
 that occur in the cycle of gaps $\pgap(\pml{13})$. We can use these as initial conditions
 for the population model in Equation~\ref{EqDynSys} of size $J \le 9$.}
 \end{center}
 \end{table}
\renewcommand\arraystretch{1}

Table~\ref{G13Table} begins to suggest that the ratios implied by Hardy and Littlewood's Conjecture B may
hold true in Eratosthenes sieve.  For the gaps $g=2, 4, 6, \ldots, 32$ the values of $w_{g,1}(\infty)$ are 
exactly equal to the factor in Conjecture B:
\[ w_{g,1}(\infty) = \sum_j w_{g,j}(\pml{13}) = \prod_{q>2, \; q | k} \frac{p-1}{p-2}. \]

The ratios discussed in this paper give the exact values of the relative frequencies of various gaps and
constellations as compared to the number of gaps $2$ at each stage of Eratosthenes sieve.
For gaps between primes, if the closures are at all fair as the sieving process continues, then 
these ratios in stages of the sieve should
also be good indicators of the relative occurrence of these gaps and constellations  
among primes.

\subsection{Rate of convergence to $w_{g,1}(\infty)$.}
From Equation~(\ref{EqDynSys}) we can determine more specifically the rate at which
the ratio $w_{g,1}(\pml{p})$ converges to the asymptotic value $w_{g,1}(\infty)$.
For this, we note that we can approximate Equation~(\ref{EqDynSys}) by a polynomial
in $a_2^k$.

Table~\ref{AjkTable} displays the calculated values of $a_j^k$ for $j=1,\ldots,9$ over the range of primes 
$p_0 =13$ and $p_k = 999,999,999,989$. Since $p_0 =13$, the products in each $a_j$ start with $p_1=17$.
In the table we can see the decay of the $a_j^k$ toward $0$, but $a_2^k$ and $a_3^k$ 
are still making significant contributions when $p_k \approx 10^{12}$.

\begin{lemma}\label{AjkLemma}
$$a_j^k \approx (a_2^k)^{j-1}.$$
\end{lemma}

\begin{proof}
\begin{eqnarray*}
a_j^k & = & \prod_{q=p_1}^{p_k} \frac{q-j-1}{q-2} \\
 & = & \prod_{q=p_1}^{p_k} \left( \frac{q-j-1}{q-2} \cdot \frac{q-j}{q-j} \cdot \frac{q-j+1}{q-j+1} \cdots \frac{q-3}{q-3} \right) \\
 & = & \prod_{q=p_1}^{p_k} \left( \frac{q-j-1}{q-j} \frac{q-j}{q-j+1} \cdots \frac{q-3}{q-2} \right) \\
 & = & \prod_{q=p_1}^{p_k} \frac{q-(j-2)-3}{q-(j-2)-2} \cdot \prod_{q=p_1}^{p_k} \frac{q-(j-3)-3}{q-(j-3)-2} \cdots 
 \prod_{q=p_1}^{p_k} \frac{q-3}{q-2} \\
 & \approx & (a_2^k)^{j-1}
\end{eqnarray*}
\end{proof}

In fact from the proof we see that $a_j^k < (a_2^k)^{j-1}$, and the approximation of each factor to $\frac{q-3}{q-2}$
is closer for larger $q$ relative to $j$.

From Lemma~\ref{AjkLemma} we can approximate Equation~(\ref{EqDynSys}) by a polynomial in the one variable
$a_2^k$.  Observing that the first coordinates of the $R$'s are $1$'s of alternating sign, we obtain the following
polynomial approximation:
\begin{eqnarray}\label{EqEigSys}
w_{g,1}({\pml{p_k}}) & = &  (L_1 \wgp{g}{\pml{p_0}}) - a_2^k (L_2 \wgp{g}{\pml{p_0}}) \nonumber \\
 & & \gap + a_3^k (L_3 \wgp{g}{\pml{p_0}}) \cdots 
  + (-1)^{J+1} a_J^k (L_J \wgp{g}{\pml{p_0}}) \nonumber \\
  & \approx & (L_1 \wgp{g}{\pml{p_0}}) - a_2^k (L_2 \wgp{g}{\pml{p_0}})  \\
  & & \gap  + (a_2^k)^2 (L_3 \wgp{g}{\pml{p_0}}) \cdots 
  + (-1)^{J+1} (a_2^k)^{J-1} (L_J \wgp{g}{\pml{p_0}}) \nonumber
 \end{eqnarray}
 
For any particular gap $g$, if we have initial conditions $\wgp{g}{\pml{p_0}}$ for a $p_0$ such that 
$g < 2 p_1$, then we can use Equation~(\ref{EqEigSys}) to estimate the convergence of $w_{g,1}(p)$
to its asymptotic value.

\renewcommand\arraystretch{1.2}
\begin{table}
\begin{center}
\begin{tabular}{ll} \hline
\multicolumn{2}{|c|}{Values of $a_j^k$ at $p_k=999,999,999,989$ for $p_0 = 13$} \\ \hline
 &
$a_2^k = 0.10206751799779$ \\
 & $a_3^k = 0.01019996897567$ \\
$ a_j^k = \prod_{q=17}^{p_k} \frac{q-j-1}{q-2} $ \biggap & $a_4^k = 0.00099592269918$ \\
 & $a_5^k = 0.00009477093531$ \\
 & $a_6^k = 0.00000876214163 $ \\
 & $a_7^k = 0.00000078408120$ \\
 & $a_8^k =  0.00000006757562$ \\
 & $a_9^k =  0.00000000557284$ \\ \hline
\end{tabular}
\caption{\label{AjkTable} Calculated values of the eigenvalues $a_j^k$ up to $p_k \approx 10^{12}$.  
If we use initial conditions
from $\pgap(\pml{13})$, then $p_0 = 13$ and the products start with $p_1=17$. 
Observe that $a_j^k \approx (a_2^k)^{j-1}$. }
\end{center}
\end{table}
\renewcommand\arraystretch{1}

\subsection{Primorial $g=\pml{5}=30$.}
As an example, we see in Table~\ref{G13Table} that the primorial $g=\pml{5}=30$ eventually becomes more numerous 
as a gap in the sieve than the gap $g=\pml{3}=6$:  $w_{30,1}(\infty) = 8/3$ vs $w_{6,1}(\infty) = 2$.
When does the crossover happen?

For $p_0=13$ we have $w_{6,1}(\pml{13})=1690$ and $w_{30,1}(\pml{13})=0$.  Using the calculated values
in Table~\ref{AjkTable} for $a_j^k$ at $\hat{p} \approx 10^{12}$, we find that
 $$ w_{6,1}(\pml{\hat{p}}) \approx 1.912 \biggap {\rm and} \biggap
w_{30,1}(\pml{\hat{p}}) \approx 1.579.$$
Even in the cycle of gaps $\pgap(\pml{p})$ for $p \approx 10^{12}$ the gap $6$ is more numerous than 
the gap $30$, primarily due to the slow decay of $a_2^k$.

For what $p$ will $w_{30,1}(\pml{p}) > w_{6,1}({\pml{p}})$?  That is, when will the gap $30$ be
more numerous in Eratosthenes sieve than the gap $6$? We can use the data from Table~\ref{G13Table}
to obtain coefficients for the polynomial approximations in Equation~(\ref{EqEigSys}) for both $g=6$ and
$g=30$ respectively.  Taking the difference and solving for $a_2^k$, we discover that
$w_{30,1}({\pml{p}}) > w_{6,1}({\pml{p}})$, that is, that the gaps $30$ will finally be more numerous
in Eratosthenes sieve than the gaps $6$, when $a_2^k < 0.06275$.

For $p_0=13$, when $p_k \approx 10^{12}$ the parameter $a_2^k \approx 0.1$, and when
$p_k \approx 10^{15}$ the parameter $a_2^k \approx 0.08$.  The decay of $a_2^k$ is so slow 
that there will still be fewer gaps $30$ than gaps $6$ in Eratosthenes sieve when $p \approx 10^{15}$.

\section{Polignac's conjecture and\\  Hardy \& Littlewood's Conjecture B}\label{PolSection}
At this point, here is what we know about the population of a gap $g$ through the cycles of gaps
$\pgap(\pml{p})$ in Eratosthenes sieve.  We need $p_0$ such that conditions of Theorem~\ref{ThmGrowth}
hold, in particular the condition $g < 2 p_1$.  
Then we need $J$ such that no constellation of length $J+1$ has sum
equal to $g$.  This is the size of system we need to consider, to apply the dynamic system 
of Equation~(\ref{EqDynSys}) to $g$ and its driving terms.  

For a given $g$, once we have $p_0$ and $J$, 
we start with the cycle $\pgap(\pml{p_0})$ to obtain counts of driving terms for $g$
from length $1$ to $J$.  From these initial conditions $w_g(\pml{p_0})$, we can
apply the model directly or through its eigenstructure, to obtain the exact populations 
of $g$ and its driving terms through all further
stages of Eratosthenes sieve.

Our progress along this line of increasing $p_0$ and $J$ is complicated primarily by our
having to construct $\pgap(\pml{p_0})$.  This cycle of gaps contains $\phi(\pml{p_0})$ elements,
which grows unmanageably large.  If we have $\pgap(\pml{p_0})$, then for every gap $g < 2 p_1$ 
we can enumerate the driving terms of various 
lengths.  We take the maximum such length as $J$.

\begin{table}[b]
\renewcommand\arraystretch{0.8}
\begin{center}
\begin{tabular}{|r|rrrrrrr|rr|} \hline
\multicolumn{1}{|c}{gap} 
& \multicolumn{7}{c}{ $n_{g,j}(31)$: driving terms of length $j$ in $\pgap(\pml{31})$ }
 & & \\ [1 ex]
\multicolumn{1}{|c}{ } & $\lil 3$ & $\lil 4$ & $\lil 5$ & $\lil 6$ & $\lil 7$ & $\lil 8$ & $\lil 9$
 & $\lil \sum \Rat{j}{g}$ & $\lil \Rat{1}{g}(\infty)$ \\ \hline 
$\lil g = 74$ & $\lil 1$ & $\lil 1206$ & $\lil 70194$ & $\lil 1550662$
 & $\lil 17523160$ & $\lil 113497678$ & $\lil 445136490$ & $\lil 1$ & $\lil 1.02857$ \\
$\lil 76$ &  & $\lil 602$ & $\lil 32194$ & $\lil 765488$
  & $\lil 9470176$ & $\lil 68041280$ & $\lil 302507798$ & $\lil 1.0588$ & $\lil 1.0588$\\
$\lil 78$ & & $\lil 292$ & $\lil 26060$ & $\lil 826426$
   & $\lil 12166908$ & $\lil 99284264$ & $\lil 489040926$ & $\lil 2.1818$  & $\lil 2.1818$ \\
$\lil 80$ &  & $\lil 2$ & $\lil 2876$ & $\lil 139926$ & $\lil 2656274$
    & $\lil 26634332$ & $\lil 159280176$  & $\lil 1.3333$  & $\lil 1.3333$\\
$\lil 82$ &  &  & $\lil 747$ & $\lil 46878$ & $\lil 1066848$
     & $\lil 12378176$ & $\lil 83484438$ & $\lil 1$ & $\lil 1.0256$ \\
$\lil 84$ & & $\lil 2$ & $\lil 1012$ & $\lil 58216$ & $\lil 1485176$
      & $\lil 18772184$ & $\lil 135450260$ & $\lil 2.4$ & $\lil 2.4$ \\
$\lil 86$ & & & $\lil 74$ & $\lil 4726$ & $\lil 147779$
       & $\lil 2453256$ & $\lil 23265268$ & $\lil 1$ & $\lil 1.0244$ \\
$\lil 88$ &  & & $\lil 2$ & $\lil 2190$ & $\lil 107182$
        & $\lil 2025910$ & $\lil 20603366$ & $\lil 1.1111$ & $\lil 1.1111$ \\
$\lil 90$ & & $\lil 8$ & $\lil 300$ & $\lil 9360$ & $\lil 195708$
         & $\lil 2829548$ & $\lil 26983182$ & $\lil 2.6667$ & $\lil 2.6667$ \\
$\lil 92$ & & & $\lil 20$ & $\lil 860$ & $\lil 26854$
          & $\lil 488854$ & $\lil 5364068$ & $\lil 1.0476$ & $\lil 1.0476$ \\
$\lil 94$  & & & $\lil 16$ & $\lil 740$ & $\lil 19740$
           & $\lil 333162$ & $\lil 3684805$ & $\lil 1$ & $\lil 1.0222$ \\
$\lil 96$ & &  & $\lil 4$ & $\lil 242$ & $\lil 9636$
            & $\lil 249610$ & $\lil 3693782$ & $\lil 2$ & $\lil 2$ \\
$\lil 98$ &  &  & & $\lil 28$ & $\lil 1482$ & $\lil 52328$
             & $\lil 968210$ & $\lil 1.2$ & $\lil 1.2$ \\
$\lil 100$ & & & & $\lil 8$ & $\lil 672$
              & $\lil 26428$ & $\lil 567560$ & $\lil 1.3333$ & $\lil 1.3333$ \\
$\lil 102$ &  & &  & & $\lil 78$ & $\lil 7042$
               & $\lil 249300$ & $\lil 2.133$ & $\lil 2.133$ \\
$\lil 104$ &  &  &  &  & $\lil 182$ & $\lil 6086$
                & $\lil 129016$ & $\lil 1.0909$ & $\lil 1.0909$ \\
$\lil 106$ &  & & &  & $\lil 16$ & $\lil 1168$
                 & $\lil 37144$ & $\lil 1$ & $\lil 1.0196$ \\
$\lil 108$ &  &  &  &  & $\lil 8$ & $\lil 1244$
                  & $\lil 44334$ & $\lil 2$ & $\lil 2$ \\
$\lil 110$ & & & & &  & $\lil 142$
                   & $\lil 7686$ & $\lil 1.4815$ & $\lil 1.4815$ \\
$\lil 112$  &  &  &  & &  & $\lil 68$
                    & $\lil 5294$ & $\lil 1.2$ & $\lil 1.2$ \\
$\lil 114$ & & & & &  & $\lil 22$
                     & $\lil 2388$ & $\lil 2.1176$ & $\lil 2.1176$ \\
$\lil 116$ &  & &  & & & $\lil 224$
                      & $\lil 4716$ & $\lil 1.0370$ & $\lil 1.0370$ \\
$\lil 118$ &  &  &  &  &  &  & $\lil 72$
                      & $\lil 1$ & $\lil 1.0175$ \\
$\lil 120$ & &  &  &  &  &  & $\lil 1012$ 
                      & $\lil 2.6667$ & $\lil 2.6667$ \\
$\lil 122$ & &  & &  &  &  & $\lil 70$
                      & $\lil 1$ & $\lil 1.0169$ \\
$\lil 124$ & &  & &  &  &  & $\lil 28$
                      & $\lil 1.0345$ & $\lil 1.0345$ \\
$\lil 126$ & &  & &  &  &  & $\lil 4$
                      & $\lil 2.4$ & $\lil 2.4$ \\
$\lil 128$ & &  & &  &  &  & 
                      & $\lil 1$ & $\lil 1$ \\
$\lil 130$ & &  & &  &  &  & 
                      & $\lil 1.4545$ & $\lil 1.4545$ \\
$\lil 132$ & &  & &  &  &  & $\lil 2$
                      & $\lil 2.2222$ & $\lil 2.2222$ \\
\hline
 \end{tabular}
 \caption{ \label{G31Table} A sample of the population data for gaps $g$ and their driving terms
 in the cycle of gaps $\pgap(\pml{31})$.  This section of the table records the data where the driving
 terms of length $9$ are running out.  For the range of gaps displayed, there are no nonzero entries
 for $j=1, \; 2$.  The last two columns list for each gap the ratio of the sum of all the driving terms
 in $\pgap(\pml{31})$ to the population of the gap $g=2$, and the asymptotic ratio.}
 \end{center}
 \end{table}
\renewcommand\arraystretch{1}

We now introduce an alternate way to obtain initial conditions for any gap $g$,
sufficient to apply Lemma~\ref{AsymLemma}.
As an analogue to Polignac's conjecture, we show that for any even number $2n$, the
gap $g=2n$ or its driving terms occur at some stage of Eratosthenes sieve, and we show
that although we can't apply the complete dynamic system, we do have enough
information to get the asymptotic result from Lemma~\ref{AsymLemma}.

{\em Polignac's Conjecture}:  For any even number $2n$, there are infinitely many prime pairs
$p_j$ and $p_{j+1}$ such the difference $p_{j+1}-p_j = 2n$.

In Theorem~\ref{PolThm} below we establish an analogue of Polignac's conjecture for 
Eratosthenes sieve, that for any number $2n$ the gap $g=2n$ occurs infinitely often
in Eratosthenes sieve, and the ratio of occurrences of this gap to the gap $2$ approaches
the ratio implied by Hardy \& Littlewood's Conjecture B:
$$ w_{g,1}(\infty) = \lim_{p \rightarrow \infty} \frac{n_{g,1}(\pml{p})}{n_{2,1}(\pml{p})}
= \prod_{q>2, \; q|g} \frac{q-1}{q-2}.$$

To obtain this result, we first consider $\Z \bmod Q$ and its cycle of gaps $\pgap(Q)$,
in which $Q$ is the product of the prime divisors of $2n$.
We then bring this back into Eratosthenes sieve by filling in the primes
missing from $Q$ to obtain a primorial $\pml{p}$.

Once we are working with $\pgap(\pml{p})$, the condition $g < 2 p_{k+1}$
may still prevent us from applying Theorem~\ref{ThmGrowth}.  
However, we are able to show that we have enough information to 
apply Lemma~\ref{AsymLemma} under the construction we are using.

\subsection{General recursion on cycles of gaps}
We need to develop a more general form of the recursion on cycles of gaps, one
that applies to creating $\pgap(qN)$ from $\pgap(N)$ for any prime $q$ and
number $N$.  We also need a variant of Lemma~\ref{Lemma2p} that does
not require the condition $g < 2 p_{k+1}$.

Let $\pgap(N)$ denote the cycle of gaps among the generators in $\Z \bmod N$, with
the first gap being that between $1$ and the next generator.  There are
$\phi(N)$ gaps in $\pgap(N)$ that sum to $N$. In our work in the preceding sections, 
we focused on Eratosthenes sieve, in which $N=\pml{p}$, the primorials. 

There is a one-to-one correspondence between generators of $\Z \bmod N$ and the gaps in 
$\pgap(N)$.  Let
$$\pgap(N) = g_1 \; g_2 \; \ldots g_{\phi(N)}.$$
Then for $k < \phi(N)$, $g_k$ corresponds to the generator $\gamma = 1+\sum_{j=1}^{k} g_j$, and since
$\sum_{j=1}^{\phi(N)} = N$, the generator $1$ corresponds to $g_{\phi(N)}$.  Moreover, since
$1$ and $N-1$ are always generators, $g_{\phi(N)}=2$.  For any generator $\gamma$, $N-\gamma$ 
is also a generator, which implies that except for the final $2$, $\pgap(N)$ is symmetric.
As a convention, we write the cycles with the first gap being from $1$ to the next generator. 

We build $\pgap(N)$ for any $N$ by introducing 
one prime factor at a time.

\begin{lemma}\label{LemmaCycle}
Given $\pgap(N)$, for a prime $q$ we construct $\pgap(qN)$ as follows:
\begin{enumerate}
\item[a)] if $q \mid N$, then we concatenate $q$ copies of $N$,
$$ \pgap(qN) = \underbrace{\pgap(N) \cdots \pgap(N)}_{q \gap {\rm copies}}$$
\item[b)] if $q \not| N$, then we build $\pgap(qN)$ in three steps:
\begin{enumerate}
\item[R1] Concatenate $q$ copies of $\pgap(N)$;
\item[R2] Close at $q$;
\item[R3] Close as indicated by the element-wise product $q * \pgap(N)$.
\end{enumerate}
\end{enumerate}
\end{lemma}

\begin{proof}
A number $\gamma$ in $\Z \bmod N$ is a generator iff $\gcd(\gamma,N)=1$. 
\begin{itemize}
\item[a)] 
Assume $q|N$.  Since $\gcd(\gamma,N)=1$, we know that $q \not| \gamma$.

For $j=0,1,\ldots,q-1$, we have 
$$\gcd(\gamma+jN, qN)= \gcd(\gamma,qN) = \gcd(\gamma,N)=1.$$
Thus $\gcd(\gamma,N)=1$ iff $\gcd(\gamma+jN,qN)=1$, and so the generators of $\Z \bmod qN$ have
the form $\gamma+jN$, and the gaps take the indicated form.
\item[b)] 
If $q \not| N$ then we first create a set of candidate generators for $\Z \bmod qN$, by
considering the set 
$$\set{\gamma+jN \st \gcd(\gamma,N)=1, \gap j=0,\ldots, q-1}.$$
For gaps, this is the equivalent of step R1, concatenating $q$ copies of $\pgap(N)$.
The only prime divisor we have not accounted for is $q$; if $\gcd(\gamma+jN,q)=1$, then this
candidate $\gamma+jN$ is a generator of $\Z \bmod qN$.  So we have to remove $q$ and its
multiples from among the candidates.

We first close the gaps at $q$ itself.  We index the gaps in the $q$ concatenated copies of
$\pgap(N)$:
$$ g_1 g_2 \ldots g_{\phi(N)} \ldots g_{q\cdot \phi(N)}.$$
Recalling that the first gap $g_1$ is the gap between the generator $1$ and the next smallest
generator in $\Z \bmod N$, the candidate generators are the 
running totals $\gamma_j = 1+\sum_{i=1}^{j-1} g_i$.  We take the $j$ for which $\gamma_j = q$, 
and removing $q$ from the list of candidate generators corresponds to replacing the gaps
$g_{j-1}$ and $g_j$ with the sum $g_{j-1}+g_j$.  This completes step R2 in the construction.

To remove the remaining multiples of $q$ from among the candidate generators, we note
that any multiples of $q$ that share a prime factor with $N$ have already been removed.
We need only consider multiples of $q$ that are relatively prime to $N$; that is, we only need
to remove $q\gamma_j$ for each generator $\gamma_j$ of $\Z \bmod N$ by closing the
corresponding gaps.

We can perform these closures by working directly with the cycle of gaps $\pgap(N)$.  Since
$q\gamma_{i+1}-q\gamma_{i} = qg_{i}$, we can go from one closure to the next by tallying the
running sum from the current closure until that running sum equals $qg_{i}$.
Technically, we create a series of indices beginning with $i_0=j$ such that $\gamma_j=q$,
and thereafter $i_k=j$ for which $\gamma_j-\gamma_{i_{k-1}}=q\cdot g_k$.  To cover the cycle
of gaps under construction, which consists initially of $q$ copies of $\pgap(N)$, $k$ runs
only from $0$ to $\phi(N)$.  We note that the last interval wraps around the end of the cycle
and back to $i_0$:  $i_{\phi(N)}=i_0$.
\end{itemize}
\end{proof}

\begin{theorem}\label{ThmqN}
In step R3 of Lemma~\ref{LemmaCycle}, each possible closure in $\pgap(N)$ occurs exactly once
in constructing $\pgap(qN)$.
\end{theorem}

\begin{proof}
This is again a translation of the Chinese Remainder Theorem into this setting.

Consider each gap $g$ in $\pgap(N)$.  Since $q \not| N$, $N \bmod q \neq 0$.
Under step R1 of the construction, $g$ has $q$ images.  Let the generator corresponding
to $g$ be $\gamma$.  Then the generators corresponding to the images of $g$ under step
R1 is the set:
$$\left\{ \gamma+jN \st j=0,\ldots,q-1\right\}.$$
Since $N \bmod q \neq 0$, there is exactly one $j$ for which $(\gamma+jN) \bmod q = 0$.
For this gap $g$, a closure in R2 and R3 occurs once and only once, at the image corresponding
to the indicated value of $j$.
\end{proof}

\begin{figure}[t]
\centering
\includegraphics[width=5in]{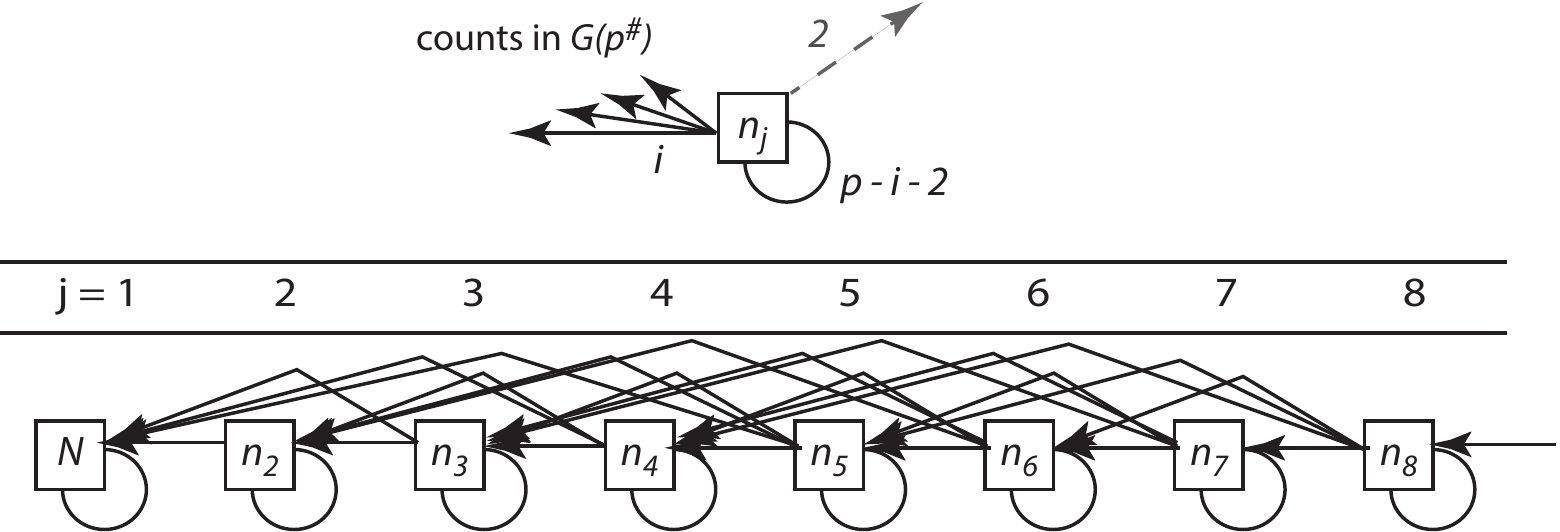}
\caption{\label{GenSystemFig} In the general dynamic system, when the 
condition $g < 2 p_{k+1}$ may not be satisfied, the interior closures may not
occur in distinct copies of the constellation.  However, the two exterior closures
still remove two copies from being driving terms for $g$.  The other $n_j-2$ copies
remain as driving terms, but we cannot specify their lengths. }
\end{figure}

\begin{corollary}\label{qNCor}
Let $g$ be a gap.  If for the prime $q$, $q \not| g$, then 
$$ \sum \Rat{j}{g}(qN) = \sum \Rat{j}{g}(N).$$
\end{corollary}

\begin{proof}
Consider a driving term $s$ for $g$, of length $j$ in $\pgap(N)$.  
In constructing $\pgap(qN)$, we initially create $q$ copies of $s$.

If $q | N$, then the construction is complete.  For each driving term for $g$
in $\pgap(N)$ we have $q$ copies, and so
$ n_{g,j}(qN) = q  \cdot n_{g,j}(N).$  However, we also have $q$ copies of
every gap $2$ in $\pgap(N)$, 
 $n_{2,1}(qN) = q \cdot n_{2,1}(N)$.
Thus $\Rat{j}{g}(qN) = \Rat{j}{g}(N)$, and we have equality for each length $j$, and 
so the result about the sum is immediate.

If $q \not| N$, then in step R1 we create $q$ copies of $s$.
In steps R2 and R3, each of the possible closures in $s$ occurs once, distributed among the
$q$ copies of $s$.  The $j-1$ closures interior to $s$ change the lengths of some of the driving
terms but don't change the sum, and the result
is still a driving term for $g$.  Only the two exterior closures, one at each end of $s$, change
the sum and thereby remove the copy from being a driving term for $g$.  Since $q \not| g$, 
these two exterior closures occur in separate copies of $s$.  See Figure~\ref{GenSystemFig}.

If the condition $g < 2p_{k+1}$ applies, then each of the closures occur in a separate copy of 
$s$, and we can use the full dynamic system of Theorem~\ref{ThmGrowth}.
For the current result we do not know that the closures necessarily occur in distinct copies of 
$s$, and so we can't be certain of the lengths of the resulting constellations.

However, we do know that of the $q$ copies of $s$, two are eliminated as driving terms and
$q-2$ remain as driving terms of various lengths.  
$$\sum_j n_{g,j}(qN) = (q-2) \sum_j n_{g,j}(N).$$
Since $n_{2,1}(qN) = (q-2)n_{2,1}(N)$, the ratios are preserved
$$\sum_j \Rat{j}{g}(qN) = \sum_j \Rat{j}{g}(N).$$
\end{proof}

As an example of this approach, we display in Table~\ref{G31Table} a sample of the enumerations 
of driving terms in $\pgap(\pml{31})$.

By combining the preceding Corollary~\ref{qNCor} with Lemma~\ref{AsymLemma},
we immediately obtain the following result, that for any gap $g$, if we look at its largest prime
factor $\bar{q}$, then we can calculate the asymptotic ratios from $\pgap(\pml{\bar{q}})$.

\begin{corollary}\label{InfCor}
Let $g=2n$ be a gap, and let $\bar{q}$ be the largest prime factor of $g$.  Then
$$ \Rat{1}{g}(\infty) = \sum \Rat{j}{g}(\pml{\bar{q}}).$$
\end{corollary}

\begin{proof}
For all primes $p > \bar{q}$, by Corollary~\ref{qNCor}
$$ \sum \Rat{j}{g}(\pml{p}) = \sum \Rat{j}{g}(\pml{\bar{q}}),$$
so once we reach $\pgap(\pml{\bar{q}})$, we continue through additional stages of the sieve
if necessary until the condition $g < 2p_1$ is satisfied, but the ratios remain unchanged during
this formality.  So the result from Lemma~\ref{AsymLemma} can be obtained from the
ratios determined in $\pgap(\pml{\bar{q}})$.
\end{proof}

\subsection{Polignac's conjecture for Eratosthenes sieve}
We establish an equivalent of Polignac's conjecture for Eratosthenes sieve.

\begin{theorem}\label{PolThm}
For every $n>0$, the gap $g=2n$ occurs infinitely often in Eratosthenes sieve, and the
ratio of the number of occurrences of $g=2n$ to the number of $2$'s converges asymptotically
to 
$$ \Rat{1}{2n}(\infty) = \prod_{q>2, \; q|n} \frac{q-1}{q-2}.
$$
\end{theorem}

We establish this result in two steps.  
First we find a stage of Eratosthenes sieve in which
the gap $g=2n$ has driving terms.  
Once we can enumerate the driving terms for $g$ 
in this initial stage of Eratosthenes sieve,
we can establish the asymptotic ratio of gaps $g=2n$ to the gaps $g=2$ as the sieve continues.

\begin{lemma}\label{QLem}
Let $g=2n$ be given.  Let $Q$ be the product of the primes dividing $2n$, including $2$.
$$Q = \prod_{q | 2n} q.$$
Finally, let $\bar{q}$ be the largest prime factor in $Q$.

Then in $\pgap(\pml{\bar{q}})$ the gap $g$ has driving terms, the total number 
of which satisfies
$$ \sum_j n_{g,j}(\pml{\bar{q}}) = \phi(Q) \cdot \prod_{p < \bar{q}, \; p\; \nmid \; Q} (p-2).$$
\end{lemma}

\begin{proof}
Let $n_1 = 2n / Q$.  By Lemma~\ref{LemmaCycle} the cycle of gaps $\pgap(2n)$ consists of $n_1$ concatenated copies
of $\pgap(Q)$.  In $\pgap(Q)$, there are $\phi(Q)$ driving terms for the gap $g=2n$.  To see this, start
at any gap in $\pgap(Q)$ and proceed through the cycle $n_1$ times.  The length of each of these
driving terms is initially $n_1 \cdot \phi(Q)$.

We now want to bring this back into Eratosthenes sieve.  

Let $Q_0=Q$, and let $p_1, \ldots, p_k$ be the prime factors of $\pml{\bar{q}} / Q$.
For $i=1,\ldots, k$, let $Q_i = p_i \cdot Q_{i-1}$, with $Q_k = \pml{\bar{q}}$.  
In forming $\pgap(Q_i)$ from $\pgap(Q_{i-1})$,
we apply Corollary~\ref{qNCor}.  Since $p_i \not| g$, we have
\begin{equation*}
 \sum_{j=1}^{J} n_{2n,j}(Q_i) =  (p_i - 2) \cdot \sum_{j=1}^{J} n_{2n,j}(Q_{i-1})
 \end{equation*}
Thus at $p_k$ we have
\begin{eqnarray*}
 \sum_{j=1}^{J} n_{2n,j}(\pml{\bar{q}}) &=&  \sum_{j=1}^{J} n_{2n,j}(Q_k) \gap = \gap
 (p_k - 2) \cdot \sum_{j=1}^{J} n_{2n,j}(Q_{k-1}) \\
  & = & \left( \prod_{i=1}^k (p_i-2) \right) \sum_{j=1}^J n_{2n,j}Q_0 
= \left( \prod_{i=1}^k (p_i-2) \right) \phi(Q) 
\end{eqnarray*}
\end{proof}

\begin{proof} {\bf of Theorem~\ref{PolThm}.}
Let $g=2n$ be given.  Let $Q$ be the product of the prime factors dividing $g$ and
let $\bar{q}$ be the largest prime factor of $g$.  By Lemma~\ref{QLem} we know that
in $\pgap(\pml{\bar{q}})$ there occur driving terms for $g$ if not the gap $g$ itself.
Lemma~\ref{QLem} gives the total number of these driving terms as
$$ \sum_j n_{g,j}(\pml{\bar{q}})  
 = \phi(Q) \cdot \prod_{p < \bar{q}, \; p\; \nmid \; Q} (p-2).$$
The number of gaps $2$ in $\pgap(\pml{q})$ is
 $n_{2,1}(\pml{q}) = \prod_{2<p\le q} (p-2).$
 So for the ratios we have
\begin{eqnarray*}
\sum_j \Rat{j}{g}(\pml{\bar{q}}) & = &
 \sum_j n_{g,j}(\pml{\bar{q}}) / n_{2,1}(\pml{\bar{q}}) \\
 &=& \phi(Q) / \prod_{p | Q, \; p > 2} (p-2)  = \prod_{p | Q, \; p > 2} \frac{(p-1)}{(p-2)}.
 \end{eqnarray*}
 By Corollary~\ref{qNCor} and Corollary~\ref{InfCor}, we have the result
 $$ \Rat{1}{2n}(\infty) = \prod_{p | 2n, \; p > 2} \left( \frac{p-1}{p-2}\right).$$
\end{proof}

This establishes a strong analogue of Polignac's conjecture for Eratosthenes sieve.
Not only do all even numbers appear as gaps in later stages of the sieve, but they
do so in proportions that converge to specific ratios.  

We use the gap $g=2$ as the reference point since it
has no driving terms other than the gap itself. 
The gaps for other even numbers appear in ratios to $g=2$
implicit in the work of Hardy and Littlewood \cite{HL}.
In their Conjecture B, they predict that the number of gaps $g=2n$ that
occur for primes less than $N$ is approximately
$$ 2 C_2 \frac{N}{(\log N)^2} \prod_{p \neq 2, \; p | 2n} \frac{p-1}{p-2}.$$
We cannot yet predict how many of the gaps in a stage of Eratosthenes sieve
will survive subsequent stages of the sieve to be confirmed as gaps among primes.
However, we note that for $g=2$, the product in the above formula evaluates to $1$, and the ratio
of gaps $g=2n$ to gaps $2$ is given by this product.  

We have shown in Theorem~\ref{PolThm} that this same product describes the
asymptotic ratio of occurrences of the gap $g=2n$ to the gap $2$ in $\pgap(\pml{p})$
as $p \fto \infty$.  So if the survival of gaps in the sieve to be confirmed as gaps among
primes is at all fair, then we would expect this ratio of gaps in the sieve to be preserved
among gaps between primes.

\subsection{Examples from $\pgap(\pml{31})$}
To work with Theorem~\ref{PolThm} we look at some data from $\pgap(\pml{31})$.
In Table~\ref{G31Table} we exhibit part of the table for $\pgap(\pml{31})$, that gives the counts
$n_{g,j}$ of driving terms of length $j$ (columns) for various gaps $g$ (rows).  
The last two columns give the current sum of driving terms for each gap and the 
asymptotic value from Theorem~\ref{PolThm}.

In each subsequent stage of Eratosthenes sieve, some copies of the driving terms of length $j$ will have
at least one interior closure, resulting in shorter driving terms at the next stage.  For this part of the table, 
$g \ge 2p_{k+1}$ and so more than one closure could occur within a single copy of a driving
term.  If a gap $g$ has a driving term of length $j$ in $\pgap(\pml{p_0})$, then at each ensuing stage of the sieve a
shorter driving term will be produced.  Thus the gap itself will occur in
$\pgap(\pml{p_k})$ for $k \le \min j-1$, using the shortest driving term for $g$
in $\pgap(\pml{p_0})$.

We have chosen the
part of the table at which the
driving terms through length $9$ are running out.  In this part of the table we observe
interesting patterns for the maximum gap associated with driving terms of a given length.
The driving terms of length $4$ have sums up to $90$ but none of sums $82$, $86$, or $88$.
Interestingly, although the gap $128$ is a power of $2$, in $\pgap(\pml{31})$ its 
driving terms span the lengths from $11$ to $27$; yet the gaps $g=126$ and $g=132$
already have driving terms of length $9$.

From the tabled values for $\pgap(\pml{31})$, we see that the driving term of length $3$
for $g=74$ will advance into an actual gap in two more stages of the sieve.  Thus the maximum
gap in $\pgap(\pml{41})$ is at least $74$, and the maximum gap for
$\pgap(\pml{43})$ is at least $90$.

With regard to our work above on Polignac's conjecture, note that in Table~\ref{G31Table}, the gaps
$g=74, 82, 86, 94,106, 118, 122$ have not attained their asymptotic ratios,
$$\sum_j \Rat{j}{g}(\pml{31}) \neq \Rat{1}{g}(\infty).$$
Up through $\pgap(\pml{31})$
these ratios are $1$, but for each gap, we know that this ratio will jump
to equal $\Rat{1}{g}(\infty)$ in the respective $\pgap(\pml{\bar{q}})$.  
How does the ratio transition from $1$ to the
asymptotic value?  If we look further in the data for $\pgap(\pml{31})$, we find that for
the gap $g=222$, $\sum_j \Rat{j}{222}(\pml{31})=2$ but the asymptotic value is
$\Rat{1}{222}(\infty) = 72/35.$

These gaps $g=2n$ have maximum prime divisor $\bar{q}$ greater than the prime $p$ for
the current stage of the sieve $\pgap(\pml{p})$.  From Corollary~\ref{qNCor} and the
approach to proving Lemma~\ref{QLem}, we are able to establish the following.

\begin{corollary}
Let $g=2n$, and let $Q=q_1 q_2 \cdots q_k$ be the product of the distinct prime factors of $g$,
with $q_1 < q_2 < \cdots < q_k$.  Then for $\pgap(\pml{p})$,
$$
\sum_j \Rat{j}{g}(\pml{p}) = \prod_{2 < q_i \le p} \left(\frac{q_i-1}{q_i-2}\right).
$$
\end{corollary}

\begin{proof}
Let $p=q_j$ for one of the prime factors in $Q$.  By Corollary~\ref{qNCor} these are the only
values of $p$ at which the sum of the ratios $\sum_j \Rat{j}{g}(p)$ changes.

Let $Q_j = q_1 q_2 \cdot q_j$.
In $\pgap(\pml{q_j})$, $g$ behaves like a multiple of $Q_j$. 
As in the proof of Lemma~\ref{QLem}, in $\pgap(Q_j)$ each generator begins a driving term
of sum $2n$, consisting of $2n / Q_j$ complete cycles.  There are $\phi(Q_j)$ such driving
terms.  

We complete $\pgap(\pml{q_j})$ as before by introducing the missing prime factors.  
The other prime factors do not
divide $2n$, and so by Corollary~\ref{qNCor} the sum of the ratios is unchanged by these factors.
We have our result:
$$ \sum_j \Rat{j}{g}(\pml{q_j}) = \prod_{2 < q_i \le q_j} \left(\frac{q_i-1}{q_i-2}\right).
$$
\end{proof}

For the gap itself, we know from Equation~\ref{EqEigSys} that the ratio $\Rat{1}{g}(\pml{p})$ 
converges to its asymptotic value
as quickly as $a_2^k \fto 0$.  We have observed above that this convergence is slow.

\section{Extending the model to constellations}
To this point we have been focusing on the populations of gaps and considering constellations
only as driving terms for gaps.  We now extend the population model to track the population of
a constellation $s$ of any length $j_1$ across stages of the sieve.

A few specific constellations of primes have been studied
\cite{HL, quads, GPY, GranPatterns}, and this work provides analogues for these studies within
Eratosthenes sieve.

The most remarkable result is a {\em strong Polignac result on arithmetic progressions}.
In Theorem~\ref{PolThm} we established that the equivalent of Polignac's conjecture holds for Eratosthenes 
sieve -- that every even number arises as a gap in the sieve, and its population
converges toward the ratio implied by Hardy and Littlewood's Conjecture B.
Extending that work to constellations, we 
now establish that for any even gap $g$, if $p$ is the maximum prime
such that $\pml{p} | g$ and $P$ is the next prime larger than $p$, 
then for every $2 \le j_1 < P-1$, the constellation $g,g,\ldots,g$
of length $j_1$ arises in Eratosthenes sieve.  This constellation corresponds to an
arithmetic progression of $j_1+1$ consecutive candidate primes.

For other examples of interesting constellations to track, we consider various constellations that
include twin primes.  The constellation $s = 2 \; 4\; 2$ corresponds to {\em prime quadruplets},
\cite{NicelyQuads, JLB, Rib, quads}.
We model the population of this constellation in Eratosthenes sieve, and we find that the constellation
$s = 2,10,2$ eventually occurs over twice as often.  

Longer constellations that include twins include
$s=2,10,2,10,2$ and
$$ s = 2, \; 10, \; 2, \; 10, \; 2, \; 4, \; 2, \; 10, \; 2, \; 10, \; 2.$$
This constellation of length $j_1 = 11$ corresponds to six pairs of twin primes in a span of $56$.  This constellation
arises in $\pgap(\pml{13})$ and persists; we analyze its population below.

\subsection{General model for populations of constellations. }
With the following theorem, we generalize the discrete dynamic system for gaps to 
model the populations of constellations across stages of Eratosthenes sieve.

\begin{theorem}\label{sThm}
Let $s$ be a constellation of length $j_1$ and of sum $|s|$.
Let $p_0$ be a prime such that $|s| < 2p_1$.

Let $n_j(s,\pml{p})$ be the population of driving terms of length $j$ for $s$ in $\pgap(\pml{p})$,
with $j_1 \le j \le J$ and $p \ge p_0$.  Then
\begin{eqnarray}\label{j1JEq}
n_s(\pml{p_k}) & = & \left. M_{j_1:J}\right|_{p_k} \cdot n_s(\pml{p_{k-1}}) \\
 & = & R \cdot \left. \Lambda_{j_1:J} \right|_{p_k} \cdot L \cdot n_s(\pml{p_{k-1}}). \notag
\end{eqnarray}
in which $L$ is the upper triangular Pascal matrix, $R$ is the upper triangular Pascal
matrix with alternating signs $(-1)^{i+j}$, and $\Lambda$ is the diagonal matrix of
eigenvalues $(p_k-j-1)$ for $j_1 \le j \le J$.
\end{theorem}

As with our work on gaps, this explicit enumeration has two limitations.  The first limitation
is that all of the populations grow super exponentially by factors of $(p-j_1-1)$.  To address this,
we normalize the dynamic system by dividing by factors of $(p-j_1-1)$, which makes the
dominant eigenvalue equal to $1$.

The second limitation of Theorem~\ref{sThm} is the condition $|s| < 2p_1$.  Because of
this condition, to exactly model the population of a constellation we have to calculate its
initial conditions in $\pgap(\pml{p})$, for $p$ very close to $|s|$.  But $\pgap(\pml{p})$ 
 consists of $\phi(\pml{p})$ gaps, which becomes unwieldy around $p=37$:  
 $\phi(\pml{37}) \approx 1.1036 \; E12.$

The model for gaps that we developed above in Sections~\ref{GapExSection} \& 
\ref{GapSection}
relies on keeping track of the internal closures and external closures
for the driving terms for gaps.  For constellations, we need a richer concept than external closures,
and so we define the {\em boundary closures} for a driving term for a constellation.

Let $s$ be a constellation of length $j_1$, $s=g_1,g_2,\ldots,g_{j_1}$.
A driving term $\tilde{s}$ for $s$ will have the form $\tilde{s} = \tilde{s_1} \; \tilde{s_2} \ldots\tilde{s_{j_1}}$
in which $|\tilde{s_i}|=g_i$ for each $i$.

We call the
closures within an $\tilde{s_i}$ the {\em interior closures} for $\tilde{s}$, and the
exterior closures for each $\tilde{s_i}$ are the {\em boundary closures} for $\tilde{s}$.
Interior closures preserve the copy of $\tilde{s}$ as a driving term for $s$ but of
shorter length.  The $j_1+1$ boundary closures remove the copy of $\tilde{s}$
from being a driving term for $s$.

We are now ready to prove Theorem~\ref{sThm}.

\begin{proof} {\em of Theorem~\ref{sThm}:}
Let $s$ be a constellation of length $j_1$ and of sum $|s|$.
Let $p_0$ be a prime such that $|s| < 2p_1$.

Let $n_{s,j}(\pml{p})$ be the population of driving terms of length $j$ for $s$ in $\pgap(\pml{p})$,
with $j_1 \le j \le J$ and $p \ge p_0$.

Consider the recursion from $\pgap(\pml{p_{k-1}})$ to $\pgap(\pml{p_k})$, with $k \ge 1$.
Let $\tilde{s}$ be a particular occurrence of a driving term of length $j$ in $\pgap(\pml{p_{k-1}})$.
Since $|s| < 2p_k$, for the $p_k$ copies of $\tilde{s}$ initially created during step R2
of the recursion, 
the $j+1$ closures of step R3 all occur in different copies.  

Of the $p_k$ initial copies of $\tilde{s}$, the boundary closures eliminate $j_1+1$ copies
as driving terms for $s$.
The $j-j_1$ interior closures produce $j-j_1$ driving terms of length $j-1$, and 
$p_k-j-1$ copies of $\tilde{s}$ survive intact as driving terms of length $j$.

We can express this in the dynamic system:
\begin{eqnarray*}
n_s (\pml{p_k}) & = & M_{j_1:J}(p_k) \cdot n_s(\pml{p_{k-1}}) \\
 & = & \left[ \begin{array}{ccccc}
 p_k-j_1-1 & 1 & 0 & \cdots & 0 \\
 0 & p_k-j_1-2 & 2 & \cdots & 0 \\
 \vdots & & \ddots & \ddots & \vdots \\
 0 & \cdots & & & J-j_1 \\
 0 & \cdots & & & p_k-J-1
 \end{array} \right] \cdot n_g(\pml{p_{k-1}})  
 \end{eqnarray*}
Note that $M_{j_1:J}(p)$ is a $(J-j_1+1)\times(J-j_1+1)$ matrix.

The matrix $M_{j_1:J}(p)$ has the same eigenvectors as the matrix for gaps (of size
$J-j_1+1$).  The eigenstructure for $M_{j_1:J}(p)$ is given by
\[  M_{j_1:J}(p) = R \cdot \Lambda_{j_1:J}(p) \cdot L
\]
in which $R$ is still the upper triangular Pascal matrix with alternating signs, and $L$
is the upper triangular Pascal matrix.  These do not depend on the prime $p$. 
The eigenvalues are 
\[ \Lambda_{j_1:J}(p) =  {\rm diag} \left[ p-j_1-1, \; p-j_1-2, \ldots, p-J-1 \right].
\]
\end{proof}

Since the eigenvectors do not depend on the prime $p$, the dynamic system
can easily be expressed in terms of the initial conditions.
\begin{eqnarray}\label{nsEq}
n_s (\pml{p_k}) & = & \left. M_{j_1:J} \right|_{p_k} \cdots \left. M \right|_{p_1} \cdot n_s(\pml{p_0}) \\
 & = & R \cdot \Lambda_{j_1:J}^k  \cdot L \cdot n_s(\pml{p_0}) \nonumber
\end{eqnarray}
with
\[ \Lambda_{j_1:J}^k = {\rm diag} \left[ \prod_1^k (p_i-j_1-1), \gap  
\prod_1^k (p_i-j_1-2), \ldots, \gap \prod_1^k (p_i-J-1) \right].
\]

\subsection{Normalizing the populations}
In the large, the population of a constellation of length $j_1$ grows primarily by a factor of $(p-j_1-1)$.  So all constellations of length $j$ ultimately become more numerous than 
any constellation of length $j+1$, and comparing the asymptotic populations of constellations
of different lengths is thereby trivial.  

On the other hand, to determine the relative occurrence {\em among constellations of a given length $j_1$}, 
we divide by the factor $(p-j_1-1)$.  
We define the {\em normalized population} of a constellation $s$ of length $j_1$ and driving terms
up to length $J$ as
\begin{equation*}
w_s (\pml{p}) = \left( \prod_{q > j_1+1}^p \frac{1}{q-j_1-1} \right) \cdot n_s(\pml{p}) 
 \; = \; \frac{1}{\phi_{j_1+1}(\pml{p})} \cdot n_s(\pml{p}).
\end{equation*}
We here introduce the functions $\phi_i(\pml{p}) = \prod_{q > i}^{p}(q-i).$

{\bf Definition.}  Let $Q=q_1 q_2 \cdots q_m$ be a product of distinct primes, with 
$q_1 < q_2 < \ldots < q_m$.  We define 
\[ \phi_{i}(Q) = \prod_{q_j > i} (q_j - i).\]
Note that $\phi_1 = \phi$, the Euler totient function, over the defined domain - products of distinct primes.  For the primorials we have
\[ \phi_{i}(\pml{p}) = \prod_{q>i}^{p} (q-i). \]

For the normalized populations, the dynamic system becomes
\begin{equation}\label{wsEq}
w_s (\pml{p_k}) =  R \cdot \Lambda_{j_1:J}^k  \cdot L \cdot w_s(\pml{p_0})
\end{equation}
with
\[ \Lambda_{j_1:J}^k = {\rm diag} \left[ 1, \gap \prod_1^k \frac{p_i-j_1-2}{p_i-j_1-1}, \gap  
\prod_1^k \frac{p_i-j_1-3}{p_i-j_1-1}, \ldots, \gap \prod_1^k \frac{p_i-J-1}{p_i-j_1-1} \right].
\]
For these normalized populations, the dominant eigenvalue is now $1$.

For gaps, this normalization corresponds nicely to taking the ratio of the population of the gap under consideration to the population of the gap $g=2$.  For $j_1=2$, we can again interpret this normalization as a ratio to a reference constellation,
in this case to the populations of constellations $s=24$ or $s=42$. These constellations have no driving terms longer than $j_1$, and $n_{24}(\pml{p}) = n_{42}(\pml{p})$, so there is no ambiguity.  This correlation between the normalization and ratios to specific reference constellations consisting of $2$'s and $4$'s begins to break down at $j_1=3$ and collapses completely at $j_1=6$.  For $j_1=3$, we have two constellations to choose from, $s=242$ or $s=424$, and there is ambiguity because $n_{424}(\pml{p}) = 2\cdot n_{242}(\pml{p})$.  For $j_1=4$ we have $s=2424$ and $s=4242$, and for $j_1=5$ we have $s=42424$.  However, by looking at $\pgap(\pml{5})$ we see that there are no constellations consisting only of $2$'s and $4$'s for length $j_1 \ge 6$.  

So we see that the normalization does not necessarily provide ratios of the populations of constellations of length $j_1$ to the population of a known reference constellation.  Instead, it provides relative populations to a (perhaps hypothetical) constellation 
-- a constellation of length $j_1$ without driving terms other than the constellation itself
and with population $\phi_{j_1+1}(\pml{p})$ and therefore asymptotic ratio $w_s^\infty = 1$.

For example, for $j_1=3$ we have $p_0=5$ and $p_0-j_1-1=1$.  The symmetric
constellation $s=242$ has length $3$; its initial population $n_{242,3}(\pml{5})=1$, and it
has no additional driving terms.  So $w_{242}^\infty = 1$, and we can use $s=242$ as the reference constellation
for constellations of length $3$.  In contrast, the symmetric
constellation $s=424$ has $w_{424}^\infty = 2$, even though it too has no additional driving terms.

\subsection{Relative occurrence in the large.} 
Equation~(\ref{wsEq}) applies under the same conditions as Theorem~\ref{sThm}: that $s$
is a constellation of length $j_1$ with driving terms up to length $J$, and $p_0$ is a prime
such that $|s| < 2p_1$.  

If we can get a count $n_s(\pml{p_0})$ for $s$ and all of its driving terms in $\pgap(\pml{p_0})$
with $|s| < 2p_1$, then we can apply the eigenstructure to obtain the
asymptotic number of occurrences $w_s^\infty$ of $s$ relative to other constellations
of length $j_1$ as $p_k \fto \infty$:
\begin{eqnarray*}
w_s^\infty & = & \lim_{p_k \fto \infty} w_{s,j_1}(\pml{p_k}) \\
 & = & \lim_{p_k \fto \infty} \frac{n_{s,j_1}(\pml{p_k})}{\phi_{j_1+1}(\pml{p_k})} \\
 & = & L_1 \cdot w_s(\pml{p_0})
\end{eqnarray*}

For this asymptotic result, we can soften the requirement of needing to work with 
$p_0$ such that $|s| < 2p_1$.  For the full dynamic system to apply, this condition
guarantees that for each occurrence $\tilde{s}$ of a driving term for 
$s$ in $\pgap(\pml{p_{k-1}})$, the closures in forming $\pgap(\pml{p_k})$ occur
in distinct copies of $\tilde{s}$.

Since $L_1 = [1, 1, \ldots, 1]$, to calculate $w_s^\infty$ we only need to know
that the $j_1+1$ boundary closures for any driving term
occur in distinct copies.  The interior closures may occur in copies that have other interior closures or
that also have a boundary closure, but we
need to know that under the recursion for prime $p$, $p-j_1-1$ copies of $\tilde{s}$ survive as driving
terms for $s$.

\begin{theorem}\label{winfThm}
Let $s$ be a constellation of $j_1$ gaps
\[ s = g_1, \; g_2, \; , \ldots, \; g_{j_1}.
\]
Let $p_0$ be the highest prime that divides any of the intervals 
$g_i + \cdots + g_j$ with $1 \le i \le j \le j_1$.
Then for any $p \ge p_0$,
\[ w_s^\infty = L_1 \cdot w_s(\pml{p}).
\]
\end{theorem}

\begin{proof}
Let $\tilde{s}$ be a driving term for $s$.
From $\pgap(\pml{p_0})$ on, the closures that remove copies of $\tilde{s}$ from being
driving terms for $s$ all occur in distinct copies.  Thus for any subsequent $p_k$,
\[ \left| n_s(\pml{p_k}) \right| = (p_k - j_1-1) \left| n_s(\pml{p_{k-1}}) \right|.
\]
We cannot be certain about the lengths of all of the copies that survive as driving terms,
but we do know how the total population grows -- by exactly the factor $p_k-j_1-1$. and
so for all $p_k > p_0$, 
\[ w_s^\infty = L_1 \cdot w_s(\pml{p_k}).
\]
\end{proof}

\subsection{Constellations related to twin primes.}
Twin primes correspond to a gap $g=2$.  Here we are studying how gaps $g=2$
arise in Eratosthenes sieve, and we do not address how many of these might survive
the sieve to become gaps between primes (in this case between twin primes).

The gap $g=2$ arises in several interesting constellations.  The first, $s=242$,
corresponds to prime quadruplets.  Prime quadruplets are the densest occurrence of 
four primes in the large, two pairs of twin primes separated by a gap of $4$.
The constellation $s=242$ has no additional driving terms, $j_1=J=3$.  
We could use $p_0=3$, for which $n_{242}(\pml{3})=[1]$.  We have
\[ w_{242}^\infty = 1.
\]

\subsubsection{Constellation $s=2,10,2$}
The next constellation we consider is $s=2,10,2$, which corresponds
to two pairs of twin primes separated by
a gap of $10$.  Here $j_1=3$ and $J=4$, for the driving terms $2642$ and $2462$.
Using $p_0=7$, we count $n_{2,10,2}(\pml{7})=[2,6]$ for two occurrences of $s=2,10,2$ and
three occurrences each of the driving terms $2642$ and $2462$.
\begin{eqnarray*}
w_{2,10,2}(\pml{7}) & = & \frac{1}{\phi_4(\pml{7})} \left[ \begin{array}{c} 2 \\ 6 \end{array} \right] 
 \; = \; \left[ \begin{array}{c} 2 / 3 \\ 2 \end{array} \right] \\
 {\rm and} & & \\
 w_{2,10,2}^\infty & = & 8 /3
\end{eqnarray*}
This means that as $p_k \fto \infty$, the number of occurrences of $s=2,10,2$ 
in the sieve approaches 
$8/3$ times the number of occurrences of the constellation $242$.  Remember that these
weights $w^\infty$ are relative only to other constellations of the same length $j_1$.

\subsubsection{Constellation $s=2,10,2,10,2$}
The constellation $s=2,10,2,10,2$ corresponds to three pairs of twin primes with gaps
of $g=10$ separating them.  This constellation also contains two overlapping copies of
$2,10,2$, and two overlapping driving terms for the constellation $12,12$ to which we will
return when we look at arithmetic progressions.

For $s=2,10,2,10,2$, we have $j_1=5$ and $J=7$.  Since $|s|=26$, 
for initial conditions we have to use $\pgap(\pml{13})$.
\begin{eqnarray*}
n_{2,10,2,10,2}(\pml{13}) & = & \left[ \begin{array}{c} 52 \\ 44 \\ 48 \end{array} \right] \\
w_{2,10,2,10,2}(\pml{13}) & = & \frac{1}{\phi_6(\pml{13})}
 \left[ \begin{array}{c} 52 \\ 44 \\ 48 \end{array} \right] \; = \;
 \frac{1}{35}
 \left[ \begin{array}{c} 52 \\ 44 \\ 48 \end{array} \right]  \\
{\rm and} & & \\
w_{2,10,2,10,2}^\infty & = & 144/35.
\end{eqnarray*}
So among constellations of length $5$, the constellation $s=2,10,2,10,2$
occurs with a relative frequency of $144/35$.  For length $j_1=5$, we can use the
constellation $42424$ as a reference.  The population model shows that in the large,
the constellation $s=2,10,2,10,2$ occurs over four times as frequently as the
constellation $42424$.

\subsubsection{Constellation $s=2,10,2,10,2,4,2,10,2,10,2$}  
Along this line of inquiry into constellations that contain several $2$s, we observe that 
the following constellation occurs in $\pgap(\pml{13})$:
\[ s = 2,10,2,10,2,4,2,10,2,10,2.
\]
This is two copies of $2,10,2,10,2$ separated by a gap of $4$.  This corresponds
to six pairs of twin primes, or twelve primes total, occurring in an interval of $|s|=56$.

For this constellation $s$, we have $j_1=11$ and $J=13$.  It would theoretically be possible
for each of the $10$'s to be produced through closures, so there could be driving terms of 
length up to $15$, but when we inspect $\pgap(\pml{13})$, we find that there are two copies of
$s$, ten driving terms of length $12$, twelve driving terms of length $13$, and no driving terms
of length $14$ or $15$.

Since $|s|=56$, to use Theorem~\ref{sThm} we would need to use $p_0=23$ to employ the 
full dynamic system.  However, to obtain the asymptotic results of Theorem~\ref{winfThm}
we can use $p_0=13$.  We calculate $w_s^\infty = 24$.  This means that relative to a
perhaps hypothetical constellation of length $11$ with one occurrence in $\pgap(\pml{13})$ and no 
additional driving terms, as the sieve 
continues the constellation $s$ will occur approximately $24$ times as often.

\begin{table}
\begin{tabular}{c|rrr|rrr} \hline
$s$ & $|s|$ & $j_1$ & $J$ & $p_0$ & $n(\pml{p_0})$ & $\omega^\infty$ \\ \hline
$242$ & $8$ & $3$ & $3$ & $5$ & $[1]$ & $1$ \\
$424$ & $10$ & $3$ & $3$ & $5$ & $[2]$ & $2$ \\
$2,10,2$ & $14$ & $3$ & $4$ & $7$ & $[2,6]$ & $8/3$ \\
$42424$ & $16$ & $5$ & $5$ & $7$ & $[1]$ & $1$ \\
$2,10,2,10,2$ & $26$ & $5$ & $7$ & $13$ & $[52, 44, 48]$ & $144/35$ \\ 
$2,10,2,10,2,4,2,10,2,10,2$ & $56$ & $11$ & $13$ & $13$ & $[2, 10, 12]$ & $24$ \\
$66$ & $12$ & $2$ & $4$ & $5$ & $[0, 2, 2]$ & $2$ \\
$12,12$ & $24$ & $2$ & $6$ & $11$ & $[0, 2, 20, 48, 58]$ & $2$ \\
$666$ & $18$ & $3$ & $5$ & $7$ & $[0, 4, 2 ]$ & $2$ \\ \hline
\end{tabular}
\caption{Table of initial conditions and parameters for a few representative constellations.
The population of a constellation of length $j_1$ grows primarily by a factor of
$p-j_1-1$, so the asymptotic weights $w_s^\infty$ must be interpreted relative only
to other constellations of the same length.}
\end{table}

\subsection{Consecutive primes in arithmetic progression}
A sequence of $j_1+1$ consecutive primes in arithmetic progression corresponds 
to a constellation of $j_1$ identical gaps $g$.
By considering residues, we easily see that for a sequence of $j_1+1$ primes in arithmetic 
progression, $g$ must be divisible by every prime $p \le j_1+1$.  
So for three consecutive primes in arithmetic progression, the minimal constellation is $s=66$.  
For four consecutive primes in arithmetic progression, the minimal constellation is $s=666$, and
then for an arithmetic progression of five consecutive primes the minimal constellation is 
$s=30,30,30,30$.  

\subsubsection{Constellation $s=66$.}
The constellation $s=66$
corresponds to three consecutive primes in arithmetic progression: $p, \; p+6, \; p+12$.
Since $|s|=12$, we can still use $p_0 = 5$.  In $\pgap(\pml{5})=64242462$, we observe the
following initial conditions for $s=66$:
$$ n_{66}(\pml{5}) = \left[ \begin{array}{c}
 0 \\ 2 \\ 2 \end{array}\right].$$
For the first entry, we don't
yet have any occurrences of $s=66$.  We do have two driving terms of length three: $642$
and $246$; and two driving terms of length four: $4242$ and $2424$.

\begin{eqnarray*}
w_{66}(\pml{5}) & = & \frac{1}{\phi_3(\pml{5})}
 \left[ \begin{array}{c} 0 \\ 2 \\ 2 \end{array} \right] \\
{\rm and} & & \\
w_{66}^\infty & = & 2.
\end{eqnarray*}

\subsubsection{Constellation $s=12,12$.}
The constellation $s=12,12$ also
corresponds to three consecutive primes in arithmetic progression: $p, \; p+12, \; p+24$.
Since $|s|=24$, to apply the full dynamic system of Theorem~\ref{sThm} we use $p_0 = 11$.  
In $\pgap(\pml{11})$, we calculate the following, for driving terms from length $j_1=2$ to $J=6$:
\[ \begin{array}{lcl} n_{12,12}(\pml{11}) = \left[ \begin{array}{c}
0 \\ 2 \\ 20 \\ 48 \\ 58 \end{array}\right] &{\rm and} &
w_{12,12}(\pml{11}) =  \frac{1}{\phi_3(\pml{11})}
 \left[ \begin{array}{c} 0 \\ 2 \\ 20 \\ 48 \\ 58 \end{array} \right] 
 \end{array} 
 \]
 So for $s=12,12$ we have
\[
w_{12,12}^\infty = \frac{128}{8\cdot 4 \cdot 2} = 2.
\]
It is interesting that the asymptotic relative population of $s=12,12$ is the same as for the
constellation $s=66$.

\subsubsection{Constellation $s=666$.}
The constellation $s=666$ is the smallest constellation
corresponding to four consecutive primes in arithmetic progression.
Since $|s|=18$, we can use $p_0 = 7$.  In $\pgap(\pml{7})$, we have the
following initial conditions for $s=666$, for driving terms from length $j_1=3$ to $J=5$:
\[ \begin{array}{lcl} n_{666}(\pml{7}) = \left[ \begin{array}{c}
0 \\ 4 \\ 2 \end{array}\right] &{\rm and} &
w_{666}(\pml{7}) =  \frac{1}{\phi_4(\pml{7})}
 \left[ \begin{array}{c} 0 \\ 4 \\ 2  \end{array} \right] 
 \end{array} 
 \]
For $s=666$ we have
\[
w_{666}^\infty = \frac{6}{3} = 2.
\]
It is interesting that we again have the asymptotic relative population of $w_s^\infty=2$, 
although here it is relative to constellations of length $j_1=3$.

\subsection{Polignac result for arithmetic progressions}
Here we establish a strong Polignac result on arithmetic progressions.
In Theorem~\ref{PolThm} we established an analogue of the Polignac conjecture
for Eratosthenes sieve.
For a gap $g=2n$, we can now show that not only does the gap occur in Eratosthenes sieve,
but that every feasible repetition of $g$ as a constellation $s=g,\ldots , g$ occurs in
the sieve.  We make this precise below. 


Above, we calculated the occurrences for a few small examples.  We cannot
perform the brute force calculation for five primes in arithmetic progression.  The minimal
constellation we would need to consider is
$s=30,30,30,30$.  To apply Theorem~\ref{sThm} we would have to use $p_0=57$.  However,
we could obtain asymptotic results via Theorem~\ref{winfThm} with $p_0=5$.
The steps we would take in order to apply Theorem~\ref{winfThm} to this constellation
can be generalized to prove Theorem~\ref{PolsThm} below.

{\bf Definition.}  Let $g$ be a gap and let $p_k$ be the
largest prime such that $\pml{p_k} | g$.  Let $s= g,\ldots,g$ be a repetition of $g$ of length $j_1$.
Then the constellation $s$ is {\em feasible} iff $j_1 < p_{k+1}-1$.

If it survives subsequent stages of the sieve, a repetition $s$ of length $j_1$ corresponds to $j_1+1$ consecutive primes in arithmetic progression.

\begin{theorem}\label{PolsThm}
Let $g$ be an even number, 
and let $Q= q_1 \cdots q_m$ be the product of the distinct prime factors of $g$, with
$q_1 < \ldots < q_m$.  Let $s$ be a feasible repetition of $g$ of length $j_1$.
Then $s$ occurs in Eratosthenes sieve with asymptotic weight
\[ w_{g,\ldots,g}^\infty = \frac{\phi_1(Q)}{\phi_{j_1+1}(Q)}.
\]
\end{theorem}

\begin{proof}
We start in the cycle of gaps $\pgap(Q)$, in which $Q$ may not be a primorial. 
$\pgap(Q)$ consists of $\phi(Q)$ gaps that sum to $Q$.  So we can start with
any of the $\phi(Q)$ gaps and continue through the cycle $g/Q$ times, and
this concatenation is a driving term for $g$.  

We repeat this concatenation of cycles $\pgap(Q)$, to identify the driving terms
for $s=g,\ldots,g$, a feasible repetition of the gap $g$ of length $j_1$. 
Starting with any gap in $\pgap(Q)$, we continue through
$\pgap(Q)$ for $j_1 \cdot g / Q$ complete cycles, and this is
a driving term $\tilde{s}$ for $s$.
Since we can start from any gap in $\pgap(Q)$, we have $\phi(Q)$ driving terms for $s$.

Now that we have seeded the construction, we move from
$\pgap(Q)$ back into the cycles of gaps for the primorials, $\pgap(\pml{q_m})$.   

{\em Case 1. $Q$ is itself a primorial.}  In this case $Q=\pml{q_m}$, and the total number of driving 
terms for $s$ equals $L_1 \cdot n_s(\pml{q_m}) = \phi(Q)$.  From this, we calculate
the asymptotic ratio
\[ w_s^\infty = \frac{\phi(\pml{q_m})}{\phi_{j_1+1}(\pml{q_m})} = \frac{\phi_1(Q)}{\phi_{j_1+1}(Q)}.
\]

{\em Case 2. $Q$ is not a primorial.}  Let $\rho_1 < \rho_2 < \cdots < \rho_k$ be the
primes less than $q_m$ that are not factors of $Q$.  Let $Q_0=Q$.  For $i=1,\ldots,k$,
let $Q_i = Q_{i-1} \cdot \rho_i.$  

Let $\tilde{s}$ be a driving term for $s$ in $\pgap(Q_{i-1})$.  Under the recursion to create
$\pgap(Q_i)$, we create $\rho_i$ copies of $\tilde{s}$ in step R2.  Then in step R3 we close
gaps as indicated by the element wise product $\rho_i * \pgap(Q_{i-1})$.  

The driving term $\tilde{s}$ is composed of $j_1$ driving terms for $g$.
\[ \tilde{s} = \tilde{s_1} \; \tilde{s_2} \; \ldots \; \tilde{s_{j_1}}
\]
in which each $|\tilde{s_i}| = g$.  

The $j_1+1$ boundary closures will eliminate copies of $\tilde{s}$
from being a driving term for $s$.
All of the intervals in Theorem~\ref{winfThm} are multiples of
$g$, and these multiples themselves have prime factors entirely divisible by the factors of $g$.
That is, if $\pml{p}$ is the largest primorial that divides $g$, then for all $2 \le j \le j_1$, 
by the feasibility of $s$, the prime factors of $j$ are factors of $\pml{p}$.

Since $\rho_i \not| g$, all of the boundary closures for $\tilde{s}$ occur in different copies
of $\tilde{s}$.  Thus, of the $\rho_i$ initial copies of $\tilde{s}$, $j_1+1$ are removed as
driving terms for $s$, and the other $\rho_i-j_1-1$ copies remain as driving terms of
some length but no longer than the length of $\tilde{s}$.  

We don't know the lengths of the driving terms for $s$ in $\pgap(Q_i)$, but we do know the
total population:
\begin{eqnarray*}
L_1 \cdot n_s(Q_i) & = & (\rho_i - j_1 - 1) \cdot L_1 \cdot n_s(Q_{i-1}) \\
 & = & (\rho_i-j_1-1)(\rho_{i-1}-j_1-1)\cdots(\rho_1-j_1-1) \cdot L_1 \cdot n_s(Q_0) \\
 & = & (\rho_i-j_1-1)(\rho_{i-1}-j_1-1)\cdots(\rho_1-j_1-1) \cdot \phi(Q).
\end{eqnarray*}
Continuing this construction, we have $Q_k=\pml{q_m}$, from which
\[ L_1 \cdot n_s(Q_k) = (\rho_k-j_1-1)(\rho_{k-1}-j_1-1)\cdots(\rho_1-j_1-1) \cdot \phi(Q),
\]
and the asymptotic ratio is
\begin{equation*}
 w_s^\infty = \frac{1}{\phi_{j_1+1}(\pml{q_m})} \cdot L_1 \cdot n_s(\pml{q_m}) \; = \;
 \frac{\phi(Q)}{\phi_{j_1+1}(Q)}.
\end{equation*}
\end{proof}

\section{Surviving the sieve:  Gaps between primes}
In our work above, we obtain several exact and asymptotic results regarding the gaps and constellations
in the cycles of gaps $\pgap(\pml{p})$.   How do we translate these results into conclusions about
the gaps and constellations that occur between prime numbers?  As Eratosthenes sieve continues,
additional closures occur.  Each additional closure eliminates two gaps to produce a new gap (their sum).

We have some evidence that the recursion is a fair process.  There is an approximate uniformity
to the replication.  Each instance of a gap in $\pgap(\pml{p_k})$ is replicated $p_{k+1}$ times uniformly
spaced in step R2, and then two of these copies are removed through closures.  Also,
the parameters for the dynamic system are independent of the size of the gap; each constellation 
of length $j$ is treated the same, with the threshold condition $g < 2p_{k+1}$.  If the recursion
is a fair process, then do we expect the survival of gaps to be fair as well?

If we had a better characterization of the survival of the gaps in $\pgap(\pml{p})$, or of the 
distribution of subsequent closures across this cycle of gaps, we would be able to make 
stronger statements about what these exact results on the gaps in Eratosthenes sieve imply
about the gaps between primes.  

We have investigated three different approaches to modeling the survival of gaps in $\pgap(\pml{p})$:
\begin{itemize}
\item {\em Naive models.}  In this approach, we combine simple observations about $\pgap(\pml{p})$ to
make estimates about survival.  For example, the distribution of copies of a gap $g$ is approximately uniform
through the cycle; the cycle is symmetric; and
all of the gaps in the interval $[p_{k+1}, \; p_{k+1}^2]$ survive as gaps between primes.
\item {\em Attrition model.}  In this approach, we fix a cycle of gaps $\pgap(\pml{p_k})$ and consider the
action of the subsequent closures in this cycle as the sieve continues
for $q = p_{k+1}, p_{k+2}, \ldots, P$.  Here $P$
is the largest prime such that $P^2 < \pml{p_k}$.
\item {\em Extremal models.}  For a gap or constellation whose populations grow rapidly enough to have growing expected
values for survival, what is the probability that {\em none} will survive after some stage of the sieve?  What implications
does this have on the distribution of gaps with $\pgap(\pml{P})$ for large $P$?
\end{itemize}

We discuss these approaches briefly here.  We do not yet have any definitive results regarding survival,
but there are promising leads for further work.  There are also aspects of the recursion which we do not feel
have been completely exploited yet -- for example, the fractal character of the recursion, or the symmetry
of a cycle of gaps and the constellation of powers of $2$ at its center.

\subsection{Uniformity and naive models.}

The simplest estimates we can make regarding survivability of gaps in the sieve to become gaps between primes,
is to assume that the copies of a gap are approximately uniformly distributed throughout $\pgap(\pml{p_k})$, and
to observe that all of the gaps between $p_{k+1}$ and $p_{k+1}^2$ are in fact gaps between primes.  So our
naive estimate $E_g[p_{k+1},p_{k+1}^2]$ of the number of gaps $g$ that occur between primes in the interval
$[p_{k+1},p_{k+1}^2]$ is:
\[ E_g[p_{k+1},p_{k+1}^2] = \frac{p_{k+1}^2 - p_{k+1}}{\pml{p_k}} \cdot n_{g,1}(\pml{p_k}).
\]
The same naive estimate can be applied to constellations as well.

The assumption of uniformity is a reasonable approximation.  For any one copy of a gap $g$ in $\pgap(\pml{p_{k-1}})$,
the step R2 of the recursion will create $p_k$ images of this copy of $g$, and these images {\em are} uniformly distributed.
Then in step R3, two of these $p_k$ images will be eliminated through closures; and over all the images of copies of $g$,
these closures have to respect the symmetry of $\pgap(\pml{p_k})$.

Figure~\ref{ErrorFig} displays the percentage error of the naive estimate for a few representative gaps and constellations,
through $p_{k+1}^2 \approx 10^{12}$.  The plots display some initial transient noise, followed by systematic errors in the estimates.  It is interesting that for the constellations we have studied, the error in the naive estimate seems to depend 
primarily on the length $j$ of the constellation.

\begin{figure}[t]
\centering
\includegraphics[width=5in]{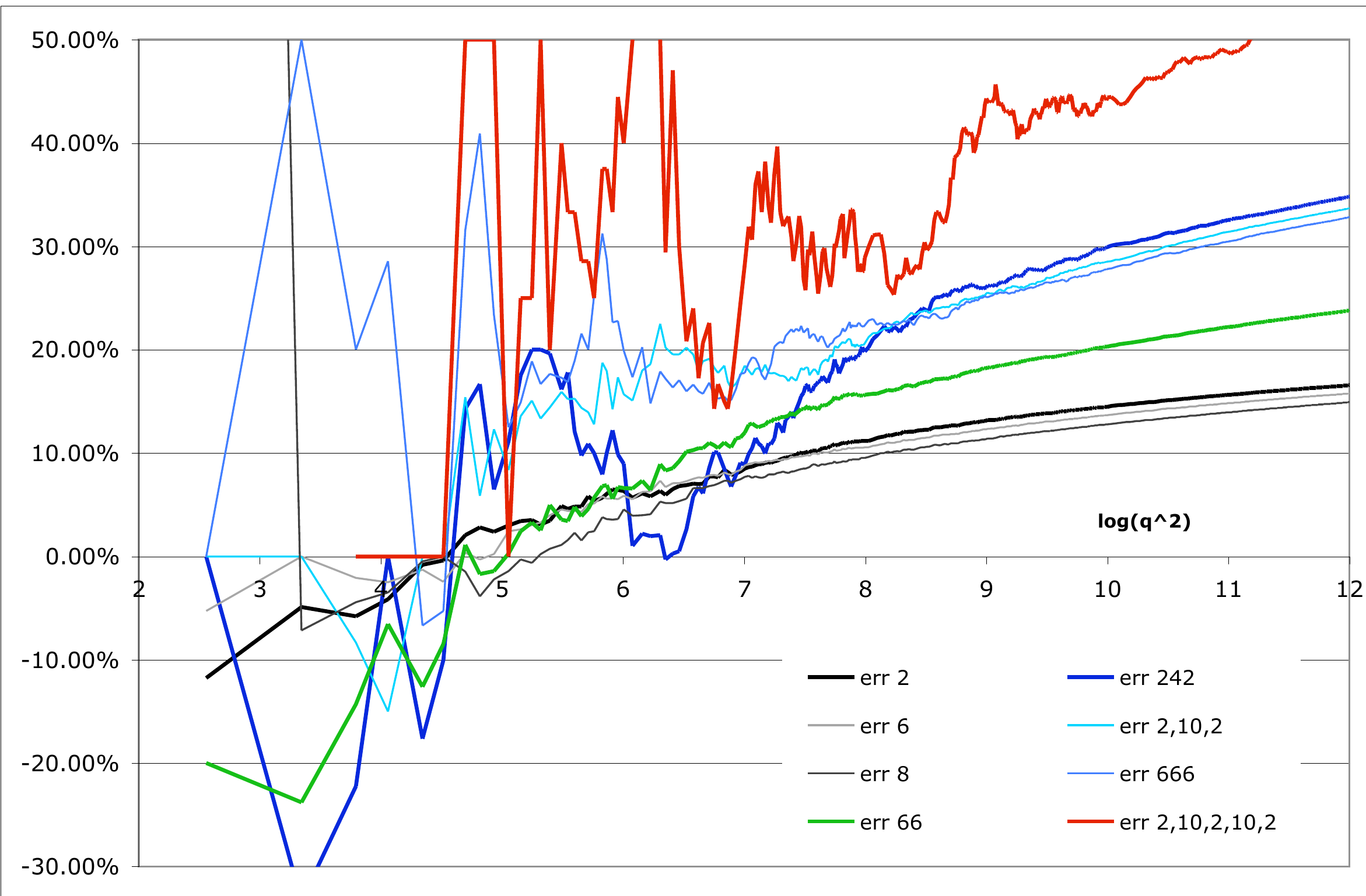}
\caption{\label{ErrorFig} A plot of the errors in the naive estimates for a few representative gaps
and constellations.  After some initial noise, the errors in the estimates appear to settle down into families
of curves which exhibit a systematic error dependent primarily on the length of the
constellation. }
\end{figure}

While the assumption of uniformity provides a decent first-order estimate of the populations of gaps and of 
short constellations between primes, we can readily identify problems with this assumption as well.
Although the population of a gap is growing superexponentially, by factors of $(p-2)$, the density of any gap $g$
in $\pgap(\pml{p})$ goes to $0$ as $p \fto \infty$; and the sample interval $[p_{k+1}, p_{k+1}^2]$ also becomes 
vanishingly small.  

We also observe that constellations persist in the sieve.  These constellations mean that there is an inherent
variation away from uniformity.  For example, a uniformly distributed gap would have average distance (say from
left side of one copy of the gap to the left side of the next copy of the gap) of
$\pml{p}/n_{g,1}(\pml{p})$.  This average distance grows arbitrarily large.
However, for $g=2$, the constellation $242$ keeps pairs of the gap as distance $6$.
The constellation $2,10,2,10,2$ keeps three copies of the gap $g=2$ at distance $12$.  Similiarly, our work
in Theorem~\ref{PolsThm} demonstrates that long repetitions of gaps occur.
This clustering of copies of a gap in short constellations means that other copies of that gap must
occur with distances much greater than the average.

One further observation about the naive estimate is that the samples are not independent.  
We appeal to steps R2 and R3 in the recursion to support a uniform approximation.  However, 
the interval $[p_{k+1}, p_{k+1}^2]$ for sampling $\pgap(\pml{p_k})$ and the next interval $[p_{k+2}, p_{k+2}^2]$
for sampling $\pgap(\pml{p_{k+1}})$ substantially overlap.  Relatively few additional closures occur between
$p_{k+1}^2$ and $p_{k+2}^2$.

\subsection{Attrition in a cycle as the sieve continues.}
Given the limitations of the naive estimate, we take a different approach toward the survival of gaps.
Let's fix a cycle of gaps $\pgap(\pml{p_k})$ and track the survival of these gaps as the sieve continues.

In $\pgap(\pml{p_k})$ the gaps at the front of the cycle, up through $p_{k+1}^2$, all survive as gaps between
primes.  Then, after closing
at $p_{k+1}^2$, the next set of gaps survive up until the closure at $p_{k+1}\cdot p_{k+2}$.
Further on in the cycle $\pgap(\pml{p_k})$ there are closures due to sieving by $p_{k+1}, p_{k+2}, \ldots, P$
where $P$ is the largest prime such that $P^2 < \pml{p_k}$.

Let's look at $\pgap(\pml{7})$ as an example. This cycle of gaps has length $48$, and the
gaps sum to $210$.
$$
\pgap(\pml{7}) =
 {\scriptstyle 
 10, 2424626424 6 62642 6 46 8 42424 8 64 6 24626 6 4246264242, 10, 2}
 $$
The first gap $10$ marks the next prime, $p_{k+1}=11$.  This first gap is the accumulation of gaps
between the primes from $1$ to $p_{k+1}$.  The next several gaps will actually survive to be confirmed
as gaps between primes, since the smallest remaining closure will occur at $p_{k+1}^2 = 121$.  
The largest prime $P$ for which we have to consider additional closures is $13$; for the next prime,
$17^2 > \pml{7} = 210$.

\begin{figure}[t]
\centering
\includegraphics[width=5in]{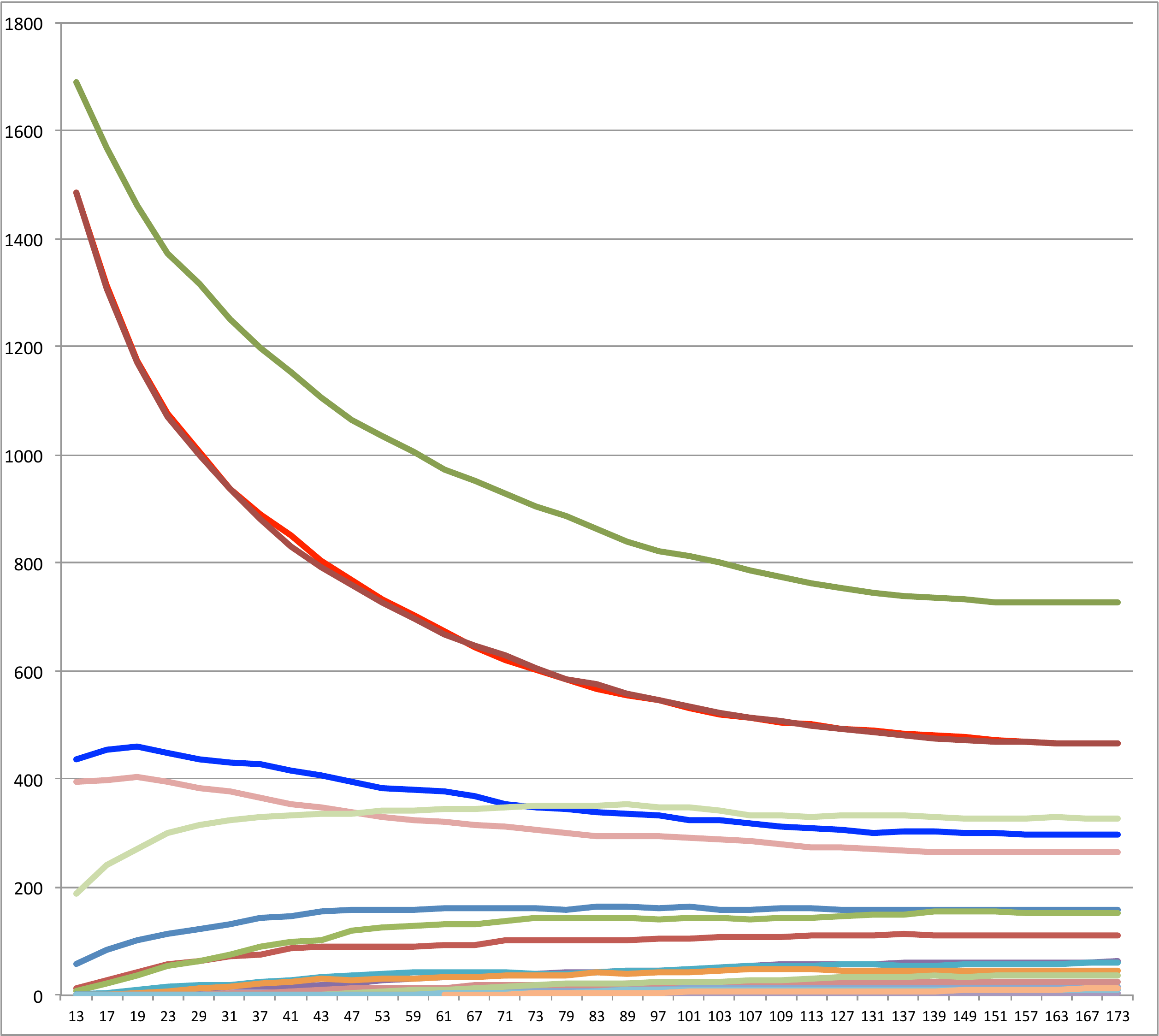}
\includegraphics[width=5in]{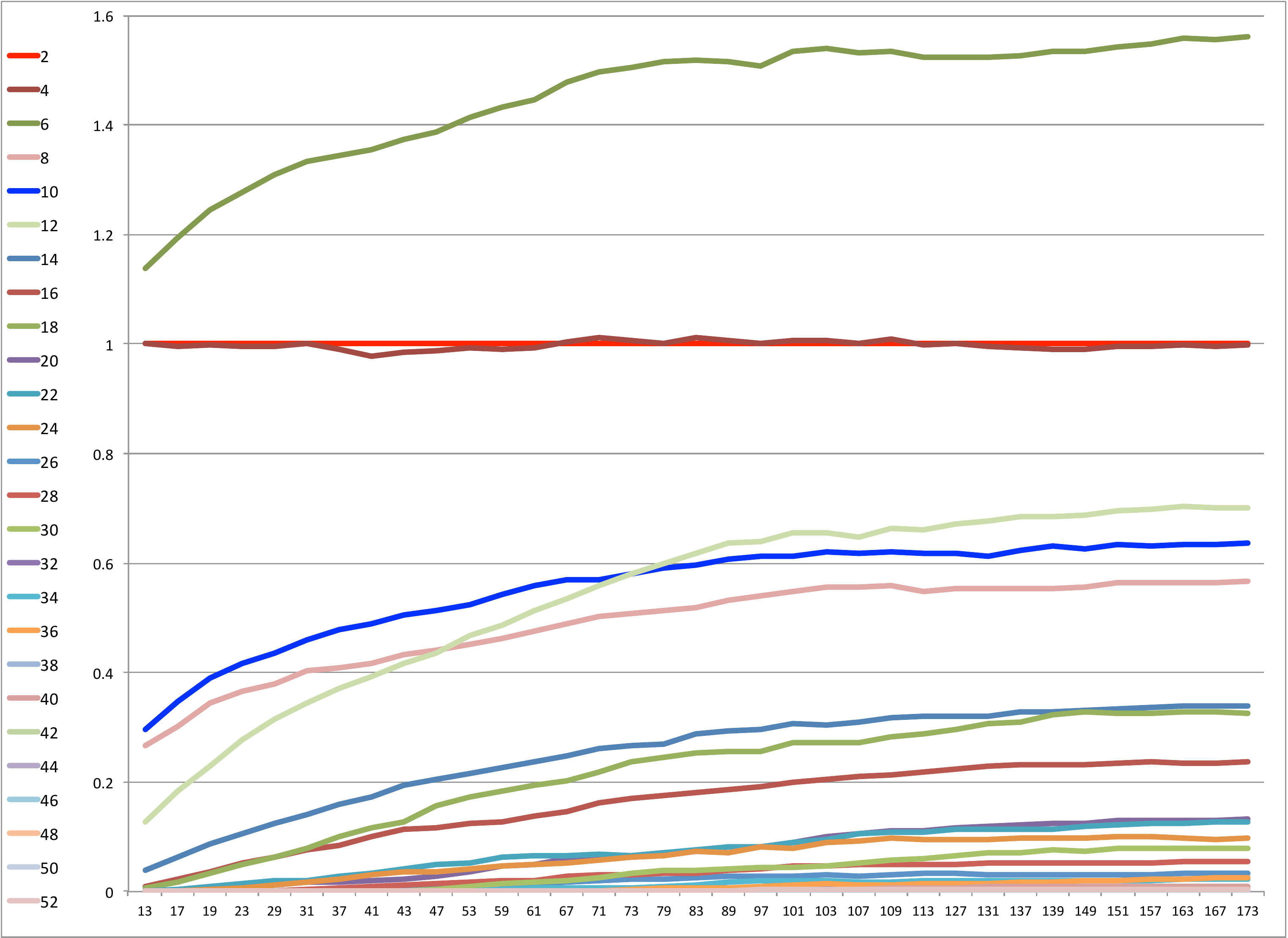}
\caption{\label{AttritionFig} A plot of the effect of attrition on the gaps in $\pgap(\pml{13})$ for closure from
$p=17,\ldots, 173$. The populations of all the gaps that arise through this process are depicted in the top graph.
The lower graph depicts the normalized populations of the gaps through this process.}
\end{figure}

Here are the closures
that occur in $\pgap(\pml{7})$ as the sieve continues.  The gaps that are known to survive at each
stage are marked in bold.
$$\begin{array}{rl}
\pgap(\pml{7}) \;  = &
 {\scriptstyle 
 10, 2424626424 6 62642 6 46 8 42424 8 64 6 24626 6 4246264242, 10, 2} \\
(p=11)  \Rightarrow & {\scriptstyle 
 10, +
 \overbrace{\scriptstyle {\bf 2424626424 6 62642 6 46 8 42424} \; 8}^{110} + 
 \overbrace{\scriptstyle 6 \; {\bf 4 6 2} \; 4}^{22} +
 \overbrace{\scriptstyle 6 \; 26 6 42462 \; 6}^{44} +
 \overbrace{\scriptstyle 4 \; 242, \; 10,}^{22} + 2  \ldots} \\
(p= 13)  \Rightarrow &  {\scriptstyle 12, +
 \overbrace{\scriptstyle {\bf 424626424 6 62642 6 46 8 42424, \; 14, \;  462, \; 10, \; 26 6 4} \; 2}^{156} + 
 4 {\bf 62, \; 10,  \; 242, \; 12,} \ldots}  \\
 (p = 17)  \Rightarrow &  {\scriptstyle 
 16, \; {\bf 24626424 6 62642 6 46 8 42424, \; 14, \;  462, \; 10, \; 26 6 4 \; 6\; 62, \; 10,  \; 242, \; 12,} \ldots }  
 \end{array}
 $$
From the prime $p=17$ and up, there are no more closures for this sequence of gaps.  All of the
remaining gaps survive as gaps between primes.
So the survival process for $\pgap(\pml{7})$ can be visually summarized as:
$$\begin{array}{rlcccccc}
\pgap(\pml{7}) \;  = &
 {\scriptstyle 
 {\bf 10, 24246264}} &  {\scriptstyle 24 {\bf 6 62642 6 46 8 42424}} &  {\scriptstyle 86 {\bf 462}} &  
 {\scriptstyle  46 {\bf 26 6 4}} &  {\scriptstyle  24 {\bf 62}} &  {\scriptstyle  64 {\bf 242},} &  {\scriptstyle  10, 2} \\
 \Rightarrow &  {\scriptstyle 
{\bf  10,24626424}} &  {\scriptstyle  6 {\bf 62642 6 46 8 42424},} &  {\scriptstyle 14, {\bf 462},} &  
{\scriptstyle 10, {\bf 26 6 4}} &  
{\scriptstyle  6 {\bf 62},} &  {\scriptstyle  10, {\bf 242},} &  {\scriptstyle  12,}
\end{array}
$$

The attrition for $\pgap(\pml{13})$ is illustrated in Figure~\ref{AttritionFig}.  The top figure plots the populations
of the gaps in the cycle, through the closures due to $p=17, 19, \ldots,173$.  Initially the largest gap in $\pgap(\pml{13})$
is $g=22$; the gap $g=52$ is first created in closures by $p=73$ and this continues to be the largest gap through
the rest of this process.  The original cycle $\pgap(\pml{13})$ has $5760$ gaps, and after subsequent closures we
finish with $3245$ gaps.  The lower figure plots the populations at each stage relative to the number of gaps $g=2$.

\subsection{Extremal models.}
Based on our computational experiments, illustrated in Figure~\ref{ErrorFig}, 
the naive model seems to track the populations of gaps and short 
constellations to first order.  This naive model exhibits systematic errors, but for the limited experiments we
have conducted so far, the errors exhibit less than logarithmic growth in $p_{k+1}^2$.  So we pose the extreme
question:  if there exists a prime $P$, such that a given gap $g$ no longer occurred between primes for $p > P$,
what implications does this have for the distribution of gaps in $\pgap(\pml{p})$?  And how is this sustained
through further recursions?

Since the cycle of gaps $\pgap(\pml{p_k})$ is symmetric, the copies of $g$ must get pushed away from the ends,
at least as fast as $p_{k+1}^2$ advances.

Under the naive model, the expected number of short constellations in the sample interval $[p_{k+1},p_{k+1}^2]$ 
eventually grows as the sieve proceeds.  Ignoring driving terms, the population of a constellation of length $j$ 
grows by factors of $(p_k-j-1)$; and although the length of the cycle of gaps is $\pml{p_k}$, the
sample interval grows as $p_{k+1}^2$.
\begin{eqnarray*} E_g[p_{k+1},p_{k+1}^2] & = & \frac{p_{k+1}^2 - p_{k+1}}{\pml{p_k}} \cdot n_{g,1}(\pml{p_k}) \\
 & = & \frac{p_k^2-p_k}{p_k^2-p_k} \frac{p_{k+1}^2 - p_{k+1}}{p_k \cdot \pml{p_{k-1}}} \cdot (p_k-j-1) n_{g,1}(\pml{p_{k-1}}) \\
  & = & \frac{(p_{k+1}^2-p_{k+1})(p_k-j-1)}{(p_k^2-p_k)p_k} \frac{p_k^2 - p_k}{\pml{p_{k-1}}} \cdot  n_{g,1}(\pml{p_{k-1}}) \\
  & = & \frac{(p_{k+1}^2-p_{k+1})(p_k-j-1)}{(p_k^2-p_k)p_k} E_g[p_k,p_k^2]
\end{eqnarray*}

From this relation, we can see that for large primes $p$ the expected populations for $j=1,2$ always grow, 
even when $p_k$ and $p_{k+1}$ are twin primes.  For other constellations, these expected values grow for
large primes whenever $p_{k+1}-p_k > (j+1)/2.$  One intuitive interpretation of this is that although
the density of a gap shrinks in $\pgap(\pml{p})$, the sample interval grows faster than the density shrinks.

From Theorem~\ref{PolThm} we know that every gap $g$ arises in the sieve, and its population approaches 
a known ratio to the gaps $2$, depending only on the prime factors of $g$.
So if the attrition process does have a strong bias against the gap $g$, eventually eliminating all copies of $g$
from the intervals $[p_{k+1},p_{k+1}^2]$, this bias has to be enforced by step R3 of the recursion.  

If twin primes are eliminated after some large prime $P$, this has interesting effects on the recursion.  Suppose
no gaps $g=2$ occur in the interval $[p_{k+1},p_{k+1}^2]$.  Then for the next several stages of the sieve, in step R3 the elementwise products that mark the differences between successive closures are at least $4p$.  And without $2$'s,
the smallest gap adjacent to a $4$ is a $6$.  So the closures will be relatively sparse at the front of the cycle of
gaps.

Extremal models are those models that propose a first-order deviation away from the expected values provided
by the naive model, and especially those models that propose that step R3 in the recursion is biased for or against
certain specific gaps.  The extremal models themselves affect the elementwise products in step R3, so there is
some compatibility condition implied -- that the bias produced in survival can also be sustained through ongoing
recursions.


\section{Conclusion}
By identifying structure among the gaps in each stage of Eratosthenes sieve, we 
have been able to develop an exact model for the populations of gaps and their
driving terms across stages of the sieve.  
We have developed a model for a discrete dynamic system that takes 
the initial populations of a gap $g$ and all its driving terms in a cycle of gaps 
$\pgap(\pml{p_0})$ such that $g < 2p_1$, and thereafter provides the exact
populations of this gap and its driving terms through all subsequent cycles of gaps.

All of the gaps between primes are generated out of these cycles of gaps, with the gaps at the front of 
the cycle surviving subsequent closures.  

The coefficients of this model
do not depend on the specific gap, only on the prime for each stage of the sieve.
To this extent, the the sieve is agnostic to the size of the gaps.

On the other hand, the initial conditions for the model do depend on the size
of the gap.  More precisely, the initial conditions depend on the prime factorization of
the gap.

For several conjectures about the gaps between primes, we can offer 
precise results for their analogues in the cycles of gaps across stages of Eratosthenes
sieve.  Foremost among these analogues, perhaps, is that we are able to 
affirm in Theorem~\ref{PolThm} an analogue of Polignac's conjecture 
that also supports Hardy \& Littlewood's
Conjecture B:

{\em For any even number $2n$, the gap $g=2n$ arises in Eratosthenes sieve, and
as $p \fto \infty$, the number of occurrences of the gap $g=2n$ to the gap $2$
approaches the ratio
$$w_{2n,1}(\infty) =  \prod_{q>2, \; q|n} \frac{q-1}{q-2}.
$$}
These results
provide evidence toward the original conjectures, to the extent that gaps
in stages of Eratosthenes sieve are indicative of gaps among primes themselves.

We extend the model for gaps to apply to the populations of constellations of any length.
This extension culminates in Theorem~\ref{PolsThm}, in which we show that every feasible
repetition of a gap $g$ arises in Eratosthenes sieve, and that its population approaches the
same ratio given above.  These repetitions correspond to consecutive primes in arithmetic progression.

All of the work on the populations of gaps and constellations in Eratosthenes sieve is constructive and
deterministic.  After setting up the models, exploring some examples, and stating some general results,
we then turn to considering a few preliminary models for how gaps and constellations in the sieve 
might survive subsequent stages of the sieve, to be confirmed as gaps and constellations among primes.
This work on surviving the sieve is preliminary and includes probabilistic modeling.



\bibliographystyle{amsplain}

\providecommand{\bysame}{\leavevmode\hbox to3em{\hrulefill}\thinspace}
\providecommand{\MR}{\relax\ifhmode\unskip\space\fi MR }
\providecommand{\MRhref}[2]{%
  \href{http://www.ams.org/mathscinet-getitem?mr=#1}{#2}
}
\providecommand{\href}[2]{#2}

\end{document}